\documentclass[11pt]{amsart}

\usepackage[colorlinks=true,urlcolor=blue, citecolor=red,linkcolor=blue,linktocpage,pdfpagelabels, bookmarksnumbered,bookmarksopen]{hyperref}
\usepackage[hyperpageref]{backref}
\usepackage{cleveref}
\usepackage{amsthm} %for citing inside theorem header
\usepackage{latexsym,amsmath,amssymb}
\usepackage{accents}
\usepackage[colorinlistoftodos,prependcaption,textsize=tiny]{todonotes}
\usepackage{a4wide}
\usepackage{soul}
\usepackage{mathtools} %For \chi with a much lower index (\mychi)
\usepackage{xparse} %For a new definiton of \chi with a slightly lower index
 \usepackage{esint}
  \usepackage[english, french]{babel}
\usepackage{cancel}
\usepackage{soul}  %use \st in text for strikethrough
\makeatletter
\newenvironment{abstracts}{%
  \ifx\maketitle\relax
    \ClassWarning{\@classname}{Abstract should precede
      \protect\maketitle\space in AMS document classes; reported}%
  \fi
  \global\setbox\abstractbox=\vtop \bgroup
    \normalfont\Small
    \list{}{\labelwidth\z@
      \leftmargin3pc \rightmargin\leftmargin
      \listparindent\normalparindent \itemindent\z@
      \parsep\z@ \@plus\p@
      
      \itemsep\medskipamount
    }%
}{%
  \endlist\egroup
  \ifx\@setabstract\relax \@setabstracta \fi
}

\newcommand{\abstractin}[1]{%
  \otherlanguage{#1}%
  \item[\hskip\labelsep\scshape\abstractname.]%
}
\makeatother 

\title[Fractional Sobolev isometric immersions of planar domains]{Fractional Sobolev isometric immersions\\of planar domains} % 
\author{Siran Li}
 
\address[Siran Li]{New York University -- Shanghai, Office 1146, 1555 Century Avenue, Pudong District, Shanghai, China (200122), and  NYU-ECNU Institute of Mathematical Sciences, Room 340, Geography Building, 3663 North Zhongshan Road, Shanghai, China (200062), and School of Mathematical Sciences, Shanghai Jiao Tong University, No.~6 Science Buildings,
800 Dongchuan Road, Minhang District, Shanghai, China (200240; current address)}
\email{sl4025@nyu.edu}

\author{Mohammad Reza Pakzad}
\address[Mohammad Reza Pakzad]{IAM, University of Bonn, Endenicher allee 60, Bonn 53111, Germany, and Department of Mathematics,  University of Pittsburgh, PA, USA}
 \email{pakzad@iam.uni-bonn.de, pakzad@pitt.edu}

\author{Armin Schikorra}
\address[Armin Schikorra]{Department of Mathematics,
University of Pittsburgh,
301 Thackeray Hall,
Pittsburgh, PA 15260, USA}
\email{armin@pitt.edu}

plus 6pt minus 12pt
\def\eps{\varepsilon}

\newcommand{\two}{{\RN{2}}}
\newcommand{\gl}{{\mathfrak{gl}(2)}}
\newcommand{\e}{\eps}
\newcommand{\na}{\nabla}
\newcommand{\p}{\partial}
\newcommand{\map}{\rightarrow}
\newcommand{\eucl}{{\mathfrak{e}}}
\newcommand{\g}{{\mathfrak{g}}}

\newcommand{\TO}{{\widetilde{\Omega}}}
\newcommand{\dd}{{\rm d}}

\newcommand{\G}{\Gamma}

\newcommand{\ds}{\displaystyle}

\def\leb{{\mathcal L}}

\newtheorem{theorem}{Theorem}
\newtheorem{lemma}{Lemma}
\newtheorem{corollary}[lemma]{Corollary}
\newtheorem{proposition}[lemma]{Proposition}

\newtheorem{remark}[lemma]{Remark}
\newtheorem{definition}[lemma]{Definition}

\def\diam{{\rm diam\,}}

\def\curl{{\rm curl\,}}

\def\supp{{\rm supp\,}}

\newcommand{\RN}[1]{%
  \textup{\uppercase\expandafter{\romannumeral#1}}%
}

\newcommand{\R}{\mathbb{R}}

\newcommand{\brac}[1]{\left (#1 \right )}
\newcommand{\abs}[1]{\left |#1 \right |}

 \definecolor{Green}{rgb}{0, 0.65,0}
\newcommand{\barint}{
\rule[.036in]{.12in}{.009in}\kern-.16in \displaystyle\int }

\newcommand{\barcal}{\text{$ \rule[.036in]{.11in}{.007in}\kern-.128in\int $}}

\def\mvint_#1{\mathchoice
          {\mathop{\vrule width 6pt height 3 pt depth -2.5pt
                  \kern -8pt \intop}\nolimits_{\kern -3pt #1}}%
          {\mathop{\vrule width 5pt height 3 pt depth -2.6pt
                  \kern -6pt \intop}\nolimits_{#1}}%
          {\mathop{\vrule width 5pt height 3 pt depth -2.6pt
                  \kern -6pt \intop}\nolimits_{#1}}%
          {\mathop{\vrule width 5pt height 3 pt depth -2.6pt
                  \kern -6pt \intop}\nolimits_{#1}}}

\numberwithin{lemma}{section}
 \numberwithin{equation}{section}

\makeatletter
\@namedef{subjclassname@2020}{%
  \textup{2020} Mathematics Subject Classification}
\makeatother

\newcommand{\lap}{\Delta_{\R^n}}
\newcommand{\aleq}{\precsim}

\newcommand{\aeq}{\approx}

\let\latexchi\chi
\makeatletter
\renewcommand\chi{\@ifnextchar_\sub@chi\latexchi}
\newcommand{\sub@chi}[2]{% #1 is _, #2 is the subscript
  \@ifnextchar^{\subsup@chi{#2}}{\latexchi^{}_{#2}}%
}
\newcommand{\subsup@chi}[3]{% #1 is the subscript, #2 is ^, #3 is the superscript
  \latexchi_{#1}^{#3}%
}
\makeatother

\newcommand{\Jac}{\operatorname{Jac}}

\subjclass[2020]{35D30, 46F10, 53C24, 53A05}
\keywords{Fractional Sobolev spaces, rigidity of isometric immersions, differential geometry at low regularity, distributional Jacobian determinant}
\begin{document}
%\centerline{Differential Geometry/Analysis of PDEs}
 
 \begin{abstracts}
\abstractin{english}
We discuss $C^1$ regularity and developability of isometric immersions of flat domains into $\R^3$ enjoying a local fractional Sobolev $W^{1+s, \frac2s}$ regularity for $2/3 \le s< 1  $, generalizing the known results on Sobolev  and H\"older regimes. Ingredients of the proof include analysis of the weak Codazzi-Mainardi equations  of  the isometric immersions and study of $W^{2,\frac2s}$  planar deformations with symmetric Jacobian derivative and vanishing distributional Jacobian determinant.  On the way, we also show that the distributional Jacobian determinant, conceived as an operator defined on the Jacobian matrix, behaves like determinant of gradient matrices under products by scalar functions. 
 
\end{abstracts}
 
     \selectlanguage{english}
\maketitle
 
\section{Introduction}
In this article  we prove the $C^1$ regularity and developability of isometric immersions of class $W^{1+s,p}$ of two dimensional domains $\Omega$ into $\R^3$ for $2/3 \le s <1$ and  $sp\ge 2$, thereby generalizing the  results of  \cite{P04} for the Sobolev regime $s=1$, $p\geq 2$  and of \cite{DLP20} for the H\"older regime $s>2/3, p=\infty$. The proofs are obtained by adapting the ideas or a few of the results  appearing in \cite{P04, LP17,  LS19, DLP20} to the fractional Sobolev case.  

\subsection{Background} There are several motivations to study isometric immersions of low regularity. A first one arises from the  the strong divergence in the respective behaviors of  $C^1$ and $C^2$ isometric immersions of two dimensional domains.  This phenomenon,  known as the flexibility and rigidity dichotomy, has other parallels, e.g.\@ for the solutions of the Euler equations in  fluid dynamics. We shall direct the readers to \cite{DLP20} and the references therein for a survey of the literature on the historic problem of developability in differential geometry, alongside its connections to the above mentioned dichotomy in nonlinear PDEs and convex integration and to a conjecture by Mischa Gromov \cite[Section 3.5.5.C, Open Problems 34-36]{Gr17}.    

The second   motivation stems from  the calculus of variations and nonlinear elastic plate theory.  Surfaces with $L^2$ integrable second fundamental form and the curvature functionals such as the Willmore energy have a long history in geometric analysis and calculus of variations. In the context of nonlinear elasticity, the Kirchhoff model stipulates that the deformation of a piece of paper under body forces or boundary conditions minimizes the Willmore functional subject to the isometric constraint.  In this context, and following the methods of  Kirchheim \cite{K01}, the $C^1$ regularity and developability of  isometric immersions with $L^2$ integrable second fundamental form were proved by the second author in \cite{P04}. This result has had many  applications in the nonlinear elastic plate theory, namely in proving  density of smooth isometries in the class of $W^{2,2}$ isometric immersions \cite{P04, Ho11}, in deriving and regularity analysis of the Euler-Lagrange equations for the Kirchhoff's models on plates \cite{Ho11-2, Ho13}, in derivation of plate and shell theories from 3d nonlinear elasticity via $\Gamma$-convergence  \cite{FJM06, HLP13}, in stability analysis for nonlinear plates \cite{LeMu09}, and finally in the confinement problem for unstretchable elastic sheets \cite{VWKG00, CM05}.

 The results of this paper give us the possibility to broaden the analysis by proposing similar models involving deformations of lower regularities, but with still some control on the curvature of the image surfaces. Indeed, as shown in  \Cref{2nd}, an isometric immersion u of regularity $W^{1+s, p}$ admits a second fundamental form $\two(u)$  of regularity 
$
W^{s-1, p}
$ if $1/2<s<1$, and $p\ge 2/s$. This way we can define a fractional Willmore-like curvature functional
\begin{equation*}\label{I2-frac}
{\mathcal I} (u) :=   \|\two (u)\|_{W^{s-1, p}}
\end{equation*}
on the class of such immersions. This variational model, which we can justifiably name {\it the fractional Kirchhoff  plate model},  is rather phenomenological; nevertheless, mathematically, many of the above mentioned problems on the standard model can be reformulated in this new context and explored.   As an example, it can be asked whether its minimizers will enjoy the same  regularity as those of the standard model established in  \cite{Ho11-2}; or will they  develop new types of singularities?   The results of the present article  concern the class of admissible deformations of this model in the regularity regime $s\ge 2/3$ and could pave the way for proving regularity of the minimizers in the footsteps of  \cite{Ho11-2}.

Finally, our last motivation for the study of weakly regular  isometric immersions is that it is connected to many interesting problems in nonlinear and geometric analysis: It has lead to the development of interesting methods in geometric measure theory 
and geometric function theory \cite{J10, JP17, LP17, LS19}, and as we shall see below,  to problems on the distributional Jacobian determinant, see also \cite{LS19, GO20} in this regard.

 \subsection{Main Results} 
 Our first result is complementary to Theorem~II in Pakzad \cite{P04}, which is the case $s=1$ of our \Cref{thm-develop}, and to the recent work for $u \in c^{1,\frac 23} (\Omega,\R^3) \supset C^{1,\frac{2}{3}+\eps}(\Omega,\R^3)$ by De Lellis--Pakzad \cite[Theorem 1]{DLP20}. Following \cite[Definition~1]{DLP20},  we say a $C^1$ mapping $u$ of a two dimensional domain $\Omega$ is developable if given any point $x\in \Omega$, $u$ is either  affine around $x$, or its Jacobian  derivative $\nabla u$ is constant along the connected component of the intersection of a  line passing through $x$ with $\Omega$.  See also \cite[Section 2]{DLP20}  for equivalent conditions.    We refer to Section~\ref{frac-not} for definitions and notations regarding fractional Sobolev spaces.    
  
\begin{theorem}\label{thm-develop}
Let  $\Omega \subset \R^2$ be an open set. Consider the class of $W_{loc}^{1+s, \frac 2s}$ isometric immersions:  
\[
 I^{1+s,\frac 2s}_{loc}(\Omega, \R^3) := \left \{ u \in W^{1+s,\frac{2}{s}}_{loc}(\Omega,\R^3): \quad (\nabla u)^T \nabla u = {\rm Id} \,\, \mbox{a.e.\@ in}\,\, \Omega \right \}
\]
Then any $u \in I^{1+s, \frac{2}{s}}_{loc}(\Omega,\R^3)$ with $\frac{2}{3}\leq s<1$ is $c^{1,\frac s2}$-regular  and   developable. 
\end{theorem}
\begin{remark}
Here $c^{1,\alpha}$ denotes all mappings whose derivatives of components lie locally in the little H\"older space $c^{0,\alpha}$, which is 
the closure of smooth functions in the $C^{0,\alpha}$ norm.
\end{remark}

 As a consequence we also obtain the extension of \cite[Corollary 1.1]{P04} to fractional Sobolev spaces, cf.\@ \cite{VWKG00, CM05}.

\begin{corollary}
There exists $\rho_0>0$ such that whenever $s \geq \frac{2}{3}$ there is no $W_{loc}^{1+s,\frac{2}{s}}$ isometric immersion of the $2$-dimensional disk into a three-dimensional Euclidean ball of radius $r <\rho_0$.  
\end{corollary}

Note that $\rho_0 < \frac 12$,  as the images of such immersions will always contain  segments larger than the unit segment.

\begin{remark}
The same statements hold true for isometric immersions of $W^{1+s, p}$-regularity with $s\ge 2/3$, $sp> 2$.  If $s>2/3$, $p>3$, this fact follows from \Cref{thm-develop} by the embedding of $W_{loc}^{s,p}$ into $W_{loc}^{2/3, 3}$. In the case $s= 2/3$, $p>3$ this embedding fails, but following the footsteps of \cite{DLP20},  a proof for the developability statement can be achieved, which we leave to the reader. We have concentrated on the more challenging borderline case $sp=2$.
\end{remark}
 
\begin{remark}\label{rem:sharp1}
\Cref{thm-develop} may fail for isometric immersions of  $W^{1+s, p}$-regularity if $sp < 2$. Indeed, for any $0<s\le 1$ and $p<2/s$, the $1$-homogeneous map $u: B^1\to \R^3$ expressed in the polar coordinates as  
$$u(r, \theta) := (\frac 12 r\cos (2\theta), \frac 12 r\sin (2\theta), \frac 12 \sqrt{3}r) $$ 
is a $W^{1+s, p}$ isometric immersion of the $2$-dimensional disk into $\R^3$ but has a conical singularity at the origin. It clearly does not belong to $C^1$ and fails to be developable.
\end{remark}
\begin{remark}
Following \cite{MuPa05} for $s=1, p=2$, we expect that the  isometric immersion can be shown to be $C^1$ up to the boundary if its $W^{1+s,\frac 2s}(\Omega)$ norm is finite and $\partial \Omega$ is of class $C^{1,\alpha}$ for some $\alpha>0$. This boundary regularity fails if $\p \Omega$ is merely of class $C^1$ \cite[Remark 7]{MuPa05}.    
\end{remark}

\begin{remark} 
To establish the result, directly following the arguments of \cite{DLP20} is not enough.  Indeed, observe that $u$ is a priori not even assumed to be in $C^1$. But this is not the only difficulty,  as we will explain in \Cref{sec-mol}  and \Cref{weaker-proof}.  We will hence adapt  a new approach.
 In particular,  \Cref{th:lambdajac} below is a new contribution devised to bypass the new obstacles for the case $s=2/3$.   
\end{remark}

To set up our second and third results,   we first remind  following Brezis--Nguyen \cite{BN11} that for any  domain $\Omega \subset \R^n$, and $f$ belonging to the optimal space $W_{loc}^{\frac{n-1}{n}, n}(\Omega, \R^n)$,  the  Jacobian determinant  ${\rm Jac} (f) := Det (\nabla f)$ is well-defined as a distribution in $\mathcal D'(\Omega)$,  see also Sickel--Youssfi \cite{WY99a}.  We also refer to the fundamental works on the distributional Jacobian developed by Reshetnyak \cite{R68},  Wente \cite{W69}, Ball \cite{B76}, Tartar \cite{T79}, M\"uller \cite{M90}, Coifman--Lions--Meyer--Semmes \cite{CLMS}, and Brezis--Nirenberg \cite{BNi95}. In view of  the embedding theorems for the fractional Sobolev spaces,  ${\rm Jac} (f)$ is well-defined for $f \in W_{loc}^{s, p} (\Omega, \R^n)$ if $\frac{n-1}{n} < s\le  1$ and $p\ge \frac {n^2}{ns+1}$.  In particular, it can be established  by the methods of \cite{BN11} that if $p=n/s$, $ {\rm Jac} (f) \in W_{{loc}}^{n(s-1), 1/s}{(\Omega)}$, -- the proof is explained in \cite[Lemma 1.3]{LS19}, cf.\@ \Cref{la:distjac}. 

  Our analysis establishes a connection between isometric immersions of fractional Sobolev  regularity  and deformations of plane domains  $f$ with symmetric Jacobians $\nabla f$ and vanishing distributional Jacobian determinants  ${\rm Jac} (f)$.  In particular, the developability of isometric immersions is proved using the following similar statement for these deformations.  
\begin{theorem}\label{th-main1}
Let $\Omega \subset \R^2$ be an open set. Assume that  $s \geq 2/3$  and $f \in W^{s,\frac{2}{s}}_{loc}(\Omega,\R^2)$  with  its distributional Jacobian satisfying 
\[
{\rm curl}\, f=0\quad \mbox{and} \quad  {\rm Jac} (f)  = 0 \quad \mbox{in}\,\, \mathcal D'(\Omega).
\]
Then, $f \ds \in c^{0,\frac s2}(\Omega)$ and  for any point $x\in \Omega$, $f$ is either constant around $x$, or  it  is constant along the connected component of the intersection of a  line passing through $x$ with $\Omega$. 
\end{theorem}

See similar statements to \Cref{th-main1} in \cite[Proposition 2.29]{K01} (for Lipschitz maps),  \cite[Proposition 1.1]{P04} (for $W^{1,2}$-maps), and \cite[Theorem 1.3]{LP17} for H\"older continuous maps. The continuity of any $f$ as in \Cref{th-main1} was already shown in \cite[Theorem 1.6]{LS19}.

\begin{remark}
As in \Cref{rem:sharp1}, \Cref{th-main1} fails for $f\in W^{s,p}$ with $\frac {4}{2s+1} \le p < \frac{2}{s}$, even for $s=1$. We refer to the so-called  \lq \lq fish-like example" discussed in  \cite{FM05}: Letting $c=0$, $f = \nabla u$ satisfies $\curl (f) = 0$ and $\Jac(f) = {\mathcal H}u= 0$, however $f$ is not even continuous.
\end{remark}

Another new contribution of this article, which will turn out to be crucial in  proving \Cref{thm-develop} in the critical threshold $s= \frac 23$,  directly regards the properties of distributional Jacobian determinants.  As we shall see in \Cref{weak-product} the distributional product $\lambda \nabla g$ is well-defined provided $\lambda \in L^\infty \cap W^{s,\frac{n}{s}}(\Omega)$ and $g \in W^{s,\frac{n}{s}}(\Omega;\R^n)$  if $s> 1/2$. In view of this fact, the following seemingly natural behavior of the distributional Jacobian  determinant  can be proven:

\medskip
\begin{theorem}\label{th:lambdajac}
Let $n \geq 2$, and $\Omega \subset \R^n$  be a bounded smooth domain, or $\Omega= \R^n$.  Assume that  $s \in [\frac{n}{n+1},1)$, $\lambda \in L^\infty \cap W^{s,\frac{n}{s}}(\Omega)$, $f,g \in W^{s,\frac{n}{s}}(\Omega;\R^n)$, and that    
\begin{equation}\label{eq:distnablaflnablag}
 \nabla f = \lambda \nabla g.
\end{equation}
Then for any $\phi \in C_c^\infty(\Omega)$,
\[
 \Jac(f)[\phi]=\Jac(g)[\lambda^n \phi].
\]
\end{theorem} Note that   since $s\ge \frac{n}{n+1}$,
$$
\lambda^n \phi \in W_{0}^{s, \frac ns}{(\Omega)} \hookrightarrow W_{0}^{(1-s)n, \frac {1}{1-s}}{(\Omega)} = W_{{0}0}^{(1-s)n, \frac {1}{1-s}}{(\Omega)},
$$ and so the right hand side in the above Jacobian determinant identity is well-defined. 

\medskip
 
The outline of this paper is as follows. In \Cref{frac} we begin with some preliminaries on fractional Sobolev spaces and gather some important statements to be used later in the article. In \Cref{f-develop} we discuss developability of  fractional Sobolev 2d deformations with  symmetric Jacobian derivative and vanishing distributional  Jacobian determinant. In the subsequent Sections \ref{sec-mol} and \ref{2nd}, we will set out to define a notion of 2nd fundamental form for fractional Sobolev isometries and to derive a weak version of  Codazzi-Mainardi system of equations for it.  In \Cref{sec-compnt}   the developabilty and regularity of each component are shown.  We will complete the proof of  \Cref{thm-develop} in  \Cref{sec-dev} and  present a proof of \Cref {th:lambdajac} in \Cref{natural}. In \Cref{weaker-proof},  it is briefly shown, as a side-note, how  \Cref {th:lambdajac} can be bypassed  in case $s>2/3$. In \Cref{s-ac} we introduce a notion of fractional absolute continuity in order to give a simple proof of the known fact from \cite{HH15} that the image of a $W^{s,p}(\R, \R^2)$ deformation is of Lebesgue measure zero provided  $s>1/2$ and $sp>1$.

\subsection*{Acknowledgments.}
The authors  would like to thank Ha\"im Brezis for his helpful comments on the early draft of this manuscript. This project was based upon work supported by the National Science Foundation and was partially funded by the Deutsche Forschungsgemeinschaft (DFG, German Research Foundation).  M.R.P. was supported by the NSF award DMS-1813738 and by the DFG  via SFB 1060 - Project Id  211504053. A.S. is funded by the NSF Career award DMS-2044898 and Simons foundation grant no 579261.

\section{Fractional Sobolev spaces, an overview and some facts}\label{frac}
\subsection{Notations}\label{frac-not}
 
We will work with the Slobodeckij or Gagliardo fractional Sobolev space, also sometimes referred to as the Besov space.
Namely for any open set $\Omega \subset \R^n$, nonnegative integer $k$,  $0<s<1$ and $1 \le  p<\infty$, we define the fractional $W^{s,p}$-seminorm of a  mapping $f\in L^1_{loc}(\Omega, \R^N)$ by  
$$
\displaystyle [f]_{W^{s,p}(\Omega)}:= \Big (\int_\Omega  \int_\Omega \frac{|f(x) - f(y)|^p}{|x-y|^{n+sp}} \,\, dx\, dy \Big ) ^{1/p}, 
$$  
and we set  for any integer $k\ge 0$  (identifying $W^{0,p}$ with $L^p$ when $k=0$), 
$$
W^{k+s,p} (\Omega):=\{ f\in W^{k,p}(\Omega): \quad [D^k f]_{W^{s,p}(\Omega)} < \infty\}, 
$$ which is a Banach space with the norm  
$$
\|f\|_{W^{k+s,p}(\Omega)}:=    \|f\|_{W^{k,p}(\Omega)}   +   [D^k f]_{W^{s,p}(\Omega)}.
 $$   $W^{k+s,p}_0 (\Omega)$ is defined to be the closure of $C^\infty_c(\Omega)$ in this space.  Note that $C^\infty_c(\R^n)$ is dense  in $W^{k+s,p}(\R^n)$ \cite[Theorem 2.4]{DNPV12}.  If $\Omega$ is a bounded smooth domain, there is a bounded  linear extension operator mapping  $W^{s,p} (\Omega)$  to $W^{s,p} (\R^n)$ \cite{DNPV12,Z15}.  For any such $\Omega$, or for $\Omega=\R^n$, and $1 \le p<\infty$, these spaces coincide with the Besov-Triebel-Lizorkin type spaces  $B^{s}_{p,p} (\Omega)= F^s_{p,p}(\Omega)$ according to \cite[Proposition 2.1.2 and Section 2.4]{RuSi}. Indeed the identification can be established as these spaces are the real $(s, p)$-interpolation between $L^p$ and $W^{1,p}$ spaces, see  \cite[Example 1.8]{Lun} and \cite[Theorem 6.2.4]{BeLo}.
 
When $1<p<\infty$, the Lions-Magenes Sobolev space $W^{k+s,p}_{00}(\Omega)$  introduced in \cite{LM61-62} is the  closed subspace of  $W^{k+s,p}(\R^n)$  defined by
$$
W^{k+s, p}_{00}(\Omega):= \{f\in W^{k+s,p}(\R^n):\,\, \supp f \subset \overline \Omega\},
$$ equipped with the induced semi-norm $[f]_{W_{00}^{k+s,p}(\Omega)}$ and norm  $\|f\|_{W_{00}^{k+s,p}(\Omega)}$. We refer to \cite[Section 4.3.2]{Tri} for more references and for the following properties:  $W^{k+s, p}_{00}(\Omega)$ can also be identified as the set of those elements of $W_0^{k+s ,p}(\Omega)$ whose extensions by $0$ outside of $\Omega$ belong to $W^{k+s,p}(\R^n)$.  $C^\infty_c(\Omega)$ is dense in $W_{00}^{k+s,p}(\Omega)$ and we have
$$
\|f\|_{W_0^{k+s ,p}(\Omega)} \aleq \|f\|_{W_{00}^{k+s,p}(\Omega)}.
$$  
If $sp \neq 1$ and $\partial \Omega$ is sufficiently regular the linear operator extending $f \in  C^\infty_c (\Omega)$ by $0$ outside of $\Omega$  to $f_0 \in W^{k+s,p}(\R^n)$  satisfies 
$$
\|f_0\|_{W^{k+s ,p}(\R^n)} \aleq \|f\|_{W^{k+s,p}(\Omega)},
$$
which implies $W^{k+s,p}_{00}(\Omega) = W^{k+s,p}_{0}(\Omega)$. If $sp=1$  this is not the case and $W^{k+s,\frac 1s}_{00}(\Omega)$ is a proper dense subspace of  $W^{k+s,\frac 1s}_{0}(\Omega)$ when $\Omega \neq \R^n$.
 
 If $\Omega$ is  a bounded smooth domain or if  $\Omega= \R^n$, we  set for $0<s<1$, $1<p<\infty$: 
 $$
W^{-s,p'}(\Omega):=  (W_{00}^{s,p}(\Omega))', 
$$ with $1/p + 1/{p'} =1$, as a subset of distributions in $\mathcal D'(\Omega)$. Our definition departs from \cite[Section 2.1.1 and Section 2.4.1]{RuSi} but by \cite[Theorem 4.8.1]{Tri}, these two definitions coincide.  Therefore the extension property is still valid for negative differentiability exponent: For a bounded smooth domain $\Omega$, and $0<s<1$, any element of $W^{-s,p}$ can be extended by a bounded linear operator to an element of $W^{-s,p}(\R^n)$ \cite[Theorem 2.4.2/2]{RuSi}. Moreover by \cite[Propostion 2.1.4/2]{RuSi}
$$
W^{s, p}(\Omega) = \{ f\in \mathcal D'(\Omega): \,\,\, f \in W^{s-1, p}(\Omega) \,\,\, \mbox{and} \,\,\, Df \in W^{s-1, p}(\Omega, \R^n)\},  
$$  with equivalence of norms
\begin{equation}\label{I-1-norm}
\|f\|_{W^{s, p}(\Omega)} \aeq  \|f\|_{W^{s-1, p}(\Omega)} + \|Df\|_{W^{s-1, p}(\Omega)}.
\end{equation} For $t>-1$, the vector valued spaces $W^{t, p}(\Omega, \R^N)$  are defined to be all $\R^N$-valued  mappings whose components lie in $W^{t, p}(\Omega)$.    We will omit the target $\R^N$ when there is no ambiguity.  

It is also useful to also define for  $0<s<1$  and $1<p<\infty$ the homogenous  norm 
\begin{equation}\label{dual}
\|f\|_{\dot{W}^{-s,p'}(\Omega)}:= \sup \left\{ f [\phi]  : \,\, \phi \in C^\infty_c(\Omega)\,\, \mbox{and}\,\,
  [\phi]_{W^{s,p}_{00} (\Omega)}  \leq 1\right\} \ge \|f\|_{W^{-s,p'}(\Omega)},
\end{equation}  where here and throughout the article $f[\phi]$ denotes the action of the distribution $f$ on $\phi$. We denote the corresponding space of finite-norm distributions by  $\dot{W}^{-s,p'}(\Omega)$, and note that $C^\infty_c(\Omega)$ is dense in $\dot{W}^{-s,p'}(\Omega)$. It follows  from \eqref{I-1-norm} through a standard scaling argument that   
\begin{equation}\label{neg-frac}
[f]_{W^{s,p}(\R^n)}  \aleq  \|D f\|_{\dot{W}^{s-1,p}(\R^n)}.  
\end{equation} 
 
We conclude our presentation of fractional Sobolev spaces by a final useful observation.  For $n\ge 2$ let the differential and integral operators 
$\lap$, $\lap^{-1}$ and the Riesz transform $\mathcal R$ be respectively defined by the Fourier symbols $|\xi|^2$, $|\xi|^{-2}$ and  $i \xi /|\xi|$. It is known that $\lap^{-1}$ is a well-defined operator and coincides (modulo a conventional sign) with the Newtonian potential on  $L^2(\R^n) \supset C^\infty_c(\R^n)$.  By a classical theorem \cite[Corollary 5.2.8]{G14} the Riesz transform is a bounded operator from $L^p(\R^n)$ into $L^p(\R^n, \R^n)$ when $1<p<\infty$.  It is a linear operator commuting with differentiation, hence, via the  interpolation property \cite[Theorem 1.6]{Lun} and a scaling argument, and in view of the fact that $\mathcal R \cdot \mathcal R f= -f$ we obtain that 
$$
[\mathcal {R}f]_{W^{s,p}(\R^n, \R^n)} \aeq [f]_{W^{s,p}(\R^n)}
$$ for any $0<s<1$ and $1<p<\infty$. An argument by duality yields the   similar  estimate 
$$
\|\mathcal {R}f\|_{\dot{W}^{-s,p'}(\R^n, \R^n)} \aeq \|f\|_{\dot{W}^{-s,p'}(\R^n)}.
$$
Combining this fact with \eqref{neg-frac}, we obtain
\begin{equation}\label{grad-lap}
[D \lap^{-1} f]_{W^{s,p}(\R^n)} \aleq   \|D (D \lap^{-1}) f\|_{\dot{W}^{s-1, p}(\R^n)}  = \|\mathcal R \otimes \mathcal R f\|_{\dot{W}^{s-1, p}(\R^n)} \aeq \|f\|_{\dot{W}^{s-1, p}(\R^n)}.
\end{equation}
 
\subsection{Mollification and commutator estimates}\label{mollif}
For a given smooth bounded domain $\Omega\subset \R^n$  we fix an extension operator and for any $f \in W^{s,p} (\Omega)$, we  still denote  its extension by $f \in W^{s,p} (\R^n)$.   Throughout the paper, we fix  a standard mollifier $\varphi \in C^\infty_c(B^1)$, $\int_{B^1} \varphi=1$. For  any mapping $f\in W^{s,p}(\Omega)$ with $\Omega$ as above, we let $f_\e$ be the mollifications of the extension $f_\e := f\ast \varphi_\e$, where $\varphi_\e (x):= \frac{1}{\e^n} \varphi(\frac x\e)$.  The following estimates, which are reminiscent of \cite{CET04, cds, DLP20} will be used in our analysis:
\begin{lemma}\label{moli-est} Let $0<s<1$, $f,g\in W^{s,p}(\Omega)$, where either $\Omega\subset \R^n$ is smooth and bounded or $\Omega=\R^n$. Then  
\begin{itemize}
\smallskip
\item[{\rm (i)}] $\|f_\e - f\|_{L^p} \le o(\e^s)$.
\smallskip
\item[{\rm (ii)}] $\forall k\ge 1$ $\|\nabla^k f_\e\|_{L^p} \le o(\e^{s-k})$.
\smallskip
\item[{\rm (iii)}] If $p\ge 2$, $\forall k\ge 0$ $\|\nabla^k (f_\e g_\e - (fg)_\e)\|_{L^{p/2}} \le o(\e^{2s-k})$, 
\smallskip
 \end{itemize}
 where the bound function $o(\cdot)$ depends on $p, \varphi$ and the extension constant of $\Omega$.
 \end{lemma}
 \begin{proof}

(i) By the extension property of smooth bounded domains it is sufficient to prove the estimates for $\Omega = \R^n$. Let  for $x,y\in \R^n$ 
$$
\delta_x f(y):= f(y-x) -f(y).  
$$  We have  by H\"older's inequality
$$
\begin{aligned}
\|f_\e - f\|^p_{L^p} & = \int_{\R^n} \Big |\int_{\R^n} \delta_x f(y) \varphi_\e (x) \, dx\Big |^p dy  = \int_{\R^n} \Big |\int_{\{|x|\le \e\}} \delta_x f(y) \varphi_\e (x) \, dx\Big |^p dy
\\ & \le \int_{\R^n}  
\Big  (\int_{\{|x|\le \e\}} (|x|^{-(s+\frac np)} |\delta_x f(y)|)^p \, dx \Big ) 
\Big (\int_{\{|x|\le \e\}} (|x|^{(s+\frac np)}|\varphi_\e (x)|)^{p'} \Big )^{\frac p{p'}}   dy   
\\ & \le C \e^{sp} \int_{\R^n} \int_{\{|x|\le \e\}}  |x|^{-(sp+n)} |f(y-x)-f(y)|^p \, dx dy \le \e^{sp}o(1), 
\end{aligned}
$$ where $\frac1p + \frac1{p'} =1$, and the last inequality is a consequence of the dominated convergence and Fubini theorems, in view of the fact that the integrand belongs to  $L^1(\R^n \times \R^n)$.  

(ii) Similarly as for (i) we write:
$$
\begin{aligned}
\|\nabla^k f_\e\|^p_{L^p} & = \int_{\R^n} \Big |\int_{\R^n}  f(y-x)   \nabla^k (\varphi_\e) (x) \, dx\Big |^p dy = \int_{\R^n} \Big |\int_{\R^n} \delta_x f(y) \e^{-k} (\nabla^k \varphi)_\e (x) \, dx\Big |^p dy   \\ & \le \int_{\R^n}  
\Big  (\int_{\{|x|\le \e\}} (|x|^{-(s+\frac np)} |\delta_x f(y)|)^p \, dx \Big ) 
\Big (\int_{\{|x|\le \e\}}   (\e^{-k} |x|^{(s+\frac np)}|(\nabla^k \varphi)_\e (x)|)^{p'} \Big )^{\frac p{p'}}   dy   
\\ & \le C \e^{(s-k)p} \int_{\R^n} \int_{\{|x|\le \e\}}  |x|^{-(sp+n)} |f(y-x)-f(y)|^p \, dx dy \le \e^{(s-k)p}o(1),
\end{aligned}
$$ which is the desired estimate. 

(iii)  First we observe that for all $k\ge 0$
\begin{equation}\label{est-delta2-x}
\begin{aligned}
\hspace{0.13in} \Big |\int_{\R^n} \delta_x f (y) \delta_x g (y) \nabla^k (\varphi_\e) (x) dx\Big |   & = 
 \Big | \int_{\R^n}  \frac {\delta_x f (y)}{|x|^{s+n/p}} \frac {\delta_x g (y)} {|x|^{s+n/p}} |x|^{2(s+ n/p)} \e^{-k} (\nabla^k \varphi)_\e (x)  dx \Big |   
  \\ & \hspace{-2.15in} \le    \Big \|\frac {\delta_x f (y)}{|x|^{s+n/p}} \frac {\delta_x g (y)} {|x|^{s+n/p}}   \Big \|_{L^{\frac p2}(\{|x|\le \e\})}    
  \Big \| |x|^{2(s+ n/p)} \e^{-k} (\nabla^k \varphi)_\e (x) \Big \|_{L^\frac{p}{p-2}(\{|x|\le \e\})}       
  \\   & \hspace{-2.15in} \le C \e^{2s-k}  \Big \|\frac {\delta_x f (y)}{|x|^{s+n/p}}\Big \|_{L^{p}(\{|x|\le \e\})} \Big \| \frac {\delta_x g (y)} {|x|^{s+n/p}} \Big \|_{L^{p}(\{|x|\le \e\})}. 
  \end{aligned}
\end{equation} For $k=0$ we write for all $y\in \R^n$: 
\begin{equation*}
(f_\e g_\e - (fg)_\e) (y)= (f_\e -f ) (g_\e -g) (y) - \int_{\R^n} \delta_x f (y) \delta_x g (y)  \varphi_\e (x)\, dx.
\end{equation*} The  $L^{p/2}$ norms of the first   term is estimated by $o(\e^{2s})$, using part  (i), (ii) and H\"older's inequality. Now,  
integrating the $\frac{p}{2}$th power of  the second term   over the parameter $y$, and  using \eqref{est-delta2-x}  will yield the $o(1)$ factor and complete the proof. 

 If $k\ge 1$, it is sufficient to note that for all $y\in \R^n$:
$$
\begin{aligned}
\nabla^k (f_\e g_\e - (fg)_\e) (y)& =    \sum_{j=0}^k   \nabla^j f_\e  \otimes \nabla^{k-j} g_\e (y) - \nabla^k( fg)_\e(y) \\ & = 
\sum_{j=1}^{k-1} \nabla^j f_\e \otimes  \nabla^{k-j} g_\e (y)  + (f_\e -f)\nabla^k g_\e (y) + (g_\e - g)\nabla^k f_\e (y)   \\ 
& -  \int_{\R^n} \delta_x f (y) \delta_x g (y) \nabla^k (\varphi_\e) (x)\, dx.
\end{aligned}
$$ The  $L^{p/2}$ norms of the terms in the first summation are estimated by $o(\e^{2s-k})$, using part (ii) and H\"older's inequality. The second and third terms are estimated using (i).  Finally, integrating its $\frac p2$th power  of the last term  and once more applying \eqref{est-delta2-x}  leads to an $o(\e^{2s-k})$ control as desired.  
\end{proof}

\begin{remark}
The estimates in \Cref{moli-est} are not optimal and seem to characterize the spaces 
$b^s_{p,\infty}$ \cite[Definition 2.1.3/1]{RuSi}, which are larger than $W^{s,p}$. We conjecture that results of the paper  can still be achieved for  the borderline space $b^{2/3}_{3,\infty}$ through the same approach. 
\end{remark}

\begin{corollary}\label{commute-convergence}
Let $s\in  (0, 1)$ and $p\ge 2$. If $f,g\in W^{s,p}(\Omega)$, where either $\Omega$ is smooth and bounded or $\Omega=\R^n$, then 
$$
\lim_{\e \to 0} \|f_\e g_\e - (fg)_\e \|_{W^{2s, \frac p2}}=0. 
$$
\end{corollary}

\begin{proof} 
The idea is to use the interpolation inequality \cite[Corollary 1.1.7]{Lun}
$$
\|h \|_{W^{\theta,\frac p2}} \aleq \|h \|^{1-\theta}_{L^{\frac p2}} \|h \|^\theta_{W^{1,\frac p2}} 
$$ for all $h\in W^{1,\frac p2}$ and $0 \le  \theta\le 1$. For $0<s\le \frac 12$, we apply \Cref{moli-est}(iii) for $k=0,1$ to $h:= f_\e g_\e - (fg)_\e$ with $\theta=2s$ to obtain:
$$
\| f_\e g_\e - (fg)_\e\|_{W^{2s, \frac p2}} \le o(\e^{(1-2s)2s + 2s(2s-1)}) = o(1). 
$$   Similarly, if $\frac 12 < s<1$, we let $h := \nabla (f_\e g_\e - (fg)_\e)$ and $\theta = 2s-1$ and we apply again \Cref{moli-est}(iii) for $k=0,1,2$, and the interpolation estimate, which together yield:
\[
\begin{array}{c}
\| f_\e g_\e - (fg)_\e)\|_{L^\frac p2}  \le o(\e^{2s}) \\  \\ \mbox{and} \\ \\  \| \nabla(f_\e g_\e - (fg)_\e)\|_{W^{2s-1, \frac p2}} \le o(\e^{(1-(2s-1))(2s-1) ) + (2s-1)(2s-2) }) = o(1).  
 \end{array}  
\] \end{proof}

We  will also need the following elementary estimate,  which in fact   states the known embedding of $W^{s, \frac ns}(\R^n)$ into VMO \cite[Section I.2, Example 2]{BNi95}:
\begin{lemma}\label{VMO}
Let $\Omega \subset \R^n$ be an open set and $f\in W^{s,\frac ns}(\Omega)$. Then for all $x\in \Omega$, 
and $\e<{\rm dist} (x, \partial \Omega)$,  
$$
\lim_{\e \to 0} \fint_{B_\e(x)}|f-f_\e(x)|^\frac ns =0.
$$ 
\end{lemma}

\begin{proof}
It is sufficient to show that 
$$
\|f - f_\e (x)\|_{L^\frac ns (B_\e(x))} \le o(\e^s),
$$ which follows from the a variant of fractional Poincar\'e inequality  which is valid for all $s\in (0,1)$ and $ 1 \le p < \infty$:
$$
\|f - f_\e (x)\|_{L^p(B_\e(x))}  \le  C \e^s [f]_{W^{s,p}(B_\e(x))},
$$ and can be proved similarly as in \cite[Proposition 2.1]{DD18}, where we have replaced the average of $f$ on the ball by $f_\e(x)$. 

Here we provide another proof.  For a fixed $x \in \Omega$ we have by \Cref{moli-est}(i)  and $p= \frac ns$:  
$$ 
\|f- f_\e (x)\|_{L^\frac ns(B_\e(x))} \le  \|f - f_\e\|_{L^\frac ns (B_\e(x))} + \|f_\e- f_\e(x)\|_{L^\frac ns (B_\e(x))} \le o(\e^s)  +  \|f_\e- f_\e(x)\|_{L^\frac ns (B_\e(x))}. 
$$ It remains to bound the second term, for which can apply the standard Poincar\'e  inequality for any $f\in L^1(\Omega)$ with the proper scaling on the ball $B_\e(x)$ 
\begin{equation}\label{mol-poin}
\|f_\e - f_\e(x)\|_{L^{\frac ns}(B_\e(x))} \le C \e \|\nabla f_\e\|_{L^{\frac ns} (B_\e(x))},
\end{equation}
to obtain, this time via \Cref{moli-est}(ii) the desired estimate.   Note that we have the right to use  $f_\e(x)$ as the normalization constant  
 since $\frac ns>n$ and $W^{1,\frac ns}$ embeds in $C^{0,{1-s}}$. 
 \end{proof}

\subsection{Distributional products in fractional Sobolev spaces} 
 
 In \Cref{2nd} we will define a notion of second fundamental form for fractional Sobolev isometries through  the first part of  the following result.  We will present a proof following the methodology of   \cite{LS20}, which then is adapted to subsequently show  the complementary second part, which, in particular, will be used in proving \Cref{th:lambdajac} in \Cref{natural}.
 
\begin{proposition}\label{weak-product}
Let $n \geq 2$, $1/2<s<1$, $f \in W^{s, \frac ns}(\R^n)$.

{\rm (i)} Let $\mu \in W^{s, \frac{n}{s}}(\R^n) \cap L^\infty(\R^n)$. Then for any $\alpha \in \{1,\ldots,n\}$, the product $\mu \partial_\alpha f$ is well-defined as a distribution on $\R^n$ and 
$$ 
 \|\mu \partial_\alpha f\|_{\dot{W}^{s-1, \frac ns}(\R^n)}   \aleq [f]_{W^{s, \frac{n}{s}}(\R^n)} ([\mu] _{W^{s, \frac{n}{s}}(\R^n)}+ \|\mu\|_{L^\infty(\R^n)}).
$$ 
{\rm (ii)} Let $\mu_k \in W^{s, \frac{n}{s}}(\R^n) \cap L^\infty(\R^n)$ with 
\[
 \sup_k\,\Big (\|\mu_k\|_{L^{\frac ns} (\R^n)}  +  \|\mu_k\|_{L^\infty(\R^n)} \Big)  < \infty.
\]  Assume moreover that  $[\mu_k]_{W^{s,\frac{n}{s}}(\R^n)} \xrightarrow{k \to \infty} 0$. Then for any $\alpha \in \{1,\ldots,n\}$,
\[
\|\mu_k \partial_\alpha f\|_{\dot{W}^{s-1, \frac{n}{s}}(\R^n)} \xrightarrow{k \to \infty} 0.
\]
 
\end{proposition}
\begin{proof}
We will first show (i).  Remember that  the harmonic extension of $f \in L^1 (\R^n) \cap L^\infty (\R^n)$ to  $\R^{n+1}_+$ is defined by the Poisson extension operator \cite[Example 2.1.13]{G14} 
\begin{equation}\label{poisson}
f^h(t,x):= C_n \int_{\R^n} \frac{t}{(|x-z|^2 + t^2)^\frac {n+1}{2}} f(z) \, dz
\end{equation}  and the operator can be extended to $W^{s, \frac ns}(\R^n)$ \cite{LS20, Ing20}.  Let  $\phi \in C^\infty_c(\R^n)$ and  let $\mu^h$, $f^h$, and $\phi^h$ be the harmonic extensions of $\mu$, $f$, and $\phi$, respectively, on to $\R^{n+1}_+$.

The one-dimensional integration by parts \cite{LS20} allows us to define 
\begin{equation}\label{eq:dual2:1} \mu \partial_\alpha f [\phi] :=-\int_{\R^{n+1}_+} \partial_{n+1}\left(\mu^h \partial_\alpha f^h\, \phi^h\right).
\end{equation} By \eqref{dual}, we are going to estimate  
\[
  \|\mu \partial_\alpha f\|_{\dot{W}^{s-1,\frac ns}(\R^n)}  = \sup \left\{\mu \partial_\alpha f [\phi] : \,  \phi\in C^\infty_c(\R^n) \,\,\mbox{and}\,\, [\phi]_{W^{1-s,\frac{n}{n-s}}(\R^n)} \leq 1\right\}.
\]
So let us fix one $\phi \in C_c^\infty(\R^n)$ with $[\phi]_{W^{1-s,\frac{n}{n-s}}(\R^n)}  \leq 1$. We bound
\[
\Big | \mu \partial_\alpha f [\phi]\Big | \aleq  \int_{\R^{n+1}_+} \left|D\mu^h\right|\, \left|Df^h\right| \left|\phi^h\right| +\left|\mu_k^h\right|\, \left|Df^h\right| \left|D\phi^h\right|,
\]
as we can always tackle the $\partial_{\alpha}$ term (which is in $\R^n$-direction) via integration by parts. Here and hereafter, $D$ is the $\R^{n+1}$-dimensional gradient. 

We first claim that
\begin{equation}\label{eq:dual2:33463}
 \int_{\R^{n+1}_+} |D\mu^h|\, |Df^h| |\phi^h| \aleq [\mu]_{W^{s, \frac{n}{s}}(\R^n)}\, [f]_{W^{s, \frac{n}{s}}(\R^n)} [\phi]_{W^{1-s, \frac{n}{n-s}}(\R^n)}
\end{equation}

We have 
\[
\begin{split}
 &\int_{\R^{n+1}_+} |D\mu^h|\, |Df^h| |\phi^h|\\
 \leq\,&\int_{\R^{n}} |\mathcal{M} \phi(x)| \int_0^\infty |D\mu^h(x,t)|\, |Df^h(x,t)|\,\dd t\,\dd x\\
 \leq\,&\int_{\R^{n}} |\mathcal{M} \phi(x)| \brac{\int_0^\infty |D\mu^h(x,t)|^2\,\dd t}^{\frac{1}{2}}\, \brac{\int_0^\infty |Df^h(x,t)|^2\,\dd t}^{\frac{1}{2}}\,\dd x\\
 \aleq\,&\|\mathcal{M} \phi\|_{L^{\frac{n}{n-1}}(\R^n)} \brac{\int_{\R^n} \brac{\int_0^\infty |D\mu^h(x,t)|^2 \,\dd t}^{\frac{2n}{2}}\dd x}^{\frac{1}{2n}} \brac{\int_{\R^n} \brac{\int_0^\infty |Df^h(t,x)|^2\,\dd t}^{\frac{2n}{2}}\dd x}^{\frac{1}{2n}}
 \end{split}
\]
Here we have used for the Hardy-Littlewood maximal function $\mathcal{M}$
\[
 |\phi^h(x,t)| \aleq \mathcal{M}\phi(x).
\]
Also recall the characterization of  the homogeneous Triebel--Lizorkin spaces (listed e.g.\@ in \cite{LS20, Ing20}):
\[
\|f\|_{\dot{F}^{\alpha}_{p,q}}\aeq \brac{\int_{\R^n} \brac{\int_0^\infty |t^{1-\frac{1}{q}-\alpha} D f^h|^q dt}^{\frac{p}{q}} dx}^{\frac{1}{p}}.
\]
So, in light of the maximal theorem, we have shown that
\[
 \int_{\R^{n+1}_+} |D\mu^h|\, |Df^h| |\phi^h| \aleq 
 \|\phi\|_{L^{\frac{n}{n-1}}(\R^n)}\, \|\mu\|_{\dot{F}^{\frac{1}{2}}_{2n,2}(\R^n)} \, \|f\|_{\dot{F}^{\frac{1}{2}}_{2n,2}(\R^n)}.
\]
Thus, we can immediately conclude \eqref{eq:dual2:33463} from the embeddings \cite[Proposition 2.2.1 and Theorem 2.2.3(ii)]{RuSi} and scaling arguments:
\[
 \|\phi\|_{L^{\frac{n}{n-1}}(\R^n)} \leq [\phi]_{W^{1-s, \frac{n}{n-s}}(\R^n)} \leq 1,
\]
\[
 \|\mu\|_{\dot{F}^{\frac{1}{2}}_{2n,2}(\R^n)}  \aleq \|\mu\|_{\dot{F}^{s}_{\frac ns, \frac ns}(\R^n)} = [\mu]_{W^{s, \frac ns}(\R^n)},
\]
\[
 \|f\|_{\dot{F}^{\frac{1}{2}}_{2n,2}(\R^n)}  \aleq [f]_{W^{s, \frac ns}(\R^n)},
\] as long as $s> 1/2$.

Next we estimate 
\begin{equation}\label{wk-prod-estim}
\begin{split}
 \int_{\R^{n+1}_+} |\mu^h| |Df^h|  |D\phi^h|   &
 \leq  \brac{\int_{\R^n} \brac{\int_0^\infty |t^{1-\frac{n-s}{n}-(1-s)} D\phi ^h|^\frac{n}{n-s} dt} dx}^{\frac{n-s}{n}} \\ & \hspace{0.2in} \brac{\int_{\R^n} \brac{\int_0^\infty |\mu^h t^{1-\frac{s}{n}-s}  Df^h|^\frac ns dt} dx}^{\frac{s}{n}} \\
 & \leq   \|\mu^h \|_{L^\infty(\R^{n+1}_+)} \underbrace{[\phi]_{W^{1-s,\frac{n}{n-s}}(\R^n)}}_{\leq 1}   [f]_{W^{s,\frac ns}(\R^n)}.
 \end{split}
\end{equation} Now it is sufficient to observe that  by the maximum principle
\[
\|\mu^h \|_{L^\infty(\R^{n+1}_+)}  \leq  \|\mu \|_{L^\infty(\R^{n})} 
\] to conclude together with \eqref{eq:dual2:1} and \eqref{eq:dual2:33463}  with 
\[
 \|\mu \partial_\alpha f\|_{\dot{W}^{s-1,\frac ns}(\R^n)} \aleq  [\mu]_{W^{s, \frac{n}{s}}(\R^n)}\, [f]_{W^{s, \frac{n}{s}}(\R^n)} 
 + \|\mu\|_{L^\infty(\R^n)}  [f]_{W^{s, \frac{n}{s}}(\R^n)},
\]   which finishes the proof of (i).
    
    (ii) does not directly follow from (i).  We first analyse the asymptotic behavior of $\mu_k$.  Note that 
 since    $W^{s, \frac sn}(\R^n)$ is reflexive,  $\mu_k$ is weakly sequentially compact in $W^{s,\frac ns}$. We  shall see that $\mu_k \rightharpoonup 0$ weakly in $W^{s, \frac ns} (\R^n)$.  Indeed, take any weakly convergent subsequence, relabelled $\mu_k$,  $\mu_k \rightharpoonup \mu \in W^{s, \frac ns} (\R^n)$. Let $B_R$ be the open ball of radius $R>0$ centered at origin in $\R^n$.  For any $R>0$,   $\mu_k|_{B_R}$ is a bounded sequence in $W^{s, \frac ns} (B_R)$ and hence by \cite[Theorem 7.1]{DNPV12},   it is precompact in $L^{n/s}(B_R)$.   Since the limit of convergent subsequences cannot be anything other than $\mu|_{B_R}$, we conclude that for each $R>0$, $\mu_k \to \mu$ strongly in $L^{\frac ns}(B_R)$, and so for some subsequence, $\mu_{k_j}$ converges almost everywhere in $B_R$ to $\mu$. This implies that
 $$
\lim_{j\to \infty}  \frac{\mu_{k_j} (x)- \mu_{k_j}(y)}{|x-y|^{2s}}  =   \frac{\mu (x)- \mu (y)}{|x-y|^{2s}} 
 $$  for almost every $(x,y) \in B_R \times B_R$.  On the other hand, $[\mu_{k_j}]_{W^{s,\frac ns}(B_R)} \le [\mu_{k_j}]_{W^{s,\frac ns}(\R^n)} \to 0$ by the main assumption, which implies, again up to a subsequence of $\mu_{k_j}$, that the same limit vanishes for almost every $(x,y) \in B_R \times B_R$.  As a consequence $\mu|_{B_R}$ must be constant for all $R>0$, and since $\mu \in W^{s,\frac ns}(\R^n)$,  we obtain that $\mu \equiv 0$ is the unique weak  accumulation point of  the original sequence $\mu_k$. We finally conclude that for all $R>0$, $\|\mu_k\|_{L^\frac ns(B_R)} \to 0$. 
    
     In order to prove (ii), we note that it is sufficient to assume $f\in C^\infty_c(\R^n)$. Indeed, let $f_j \in C^\infty (\R^n)$ be such that  $[f_j - f]_{W^{s, \frac ns}(\R^n)} \to 0$. If 
  \begin{equation}\label{for-f-smooth}
 \lim_{k \to \infty}    \|\mu_k \partial_\alpha f_j\|_{W^{s-1, \frac sn}(\R^n)} = 0, 
    \end{equation} as proved below, then 
    $$
    \|\mu_k \partial_\alpha f\|_{W^{s-1, \frac sn}(\R^n)} \le   \|\mu_k \partial_\alpha (f  - f_j) \|_{W^{s-1, \frac sn}(\R^n)} +   \|\mu_k \partial_\alpha f_j\|_{W^{s-1, \frac sn}(\R^n)},     
    $$ converges to 0 too since because of (i) and the uniform boundedness of $\mu_k$ the first term on the right hand side is arbitrarily small for large $j$.  
    
  Now we will prove \eqref{for-f-smooth}. Let $f\in C^\infty_c(\R^n)$ and assume that $\supp f$ lies in the open ball $B_\rho$ in $\R^n$. Fix a smooth cut-off function $\eta \in C^\infty_c(B_{\rho+1})$ such that $\eta \equiv 1$ on $B_\rho$. We observe that for all $k$ and for all $\phi \in C^\infty_c(\R^n)$ 
   $$
   (\eta \mu_k) \p_\alpha f [\phi] = \int_{\R^n}   (\p_\alpha f)  \eta \mu_k   \phi   = \int_{B_\rho}  (\p_\alpha f)  \eta \mu_k  \phi= \int_{B_\rho} \p_\alpha f \mu \phi = \int_{\R^n} \p_\alpha f \mu \phi = \mu \p_\alpha f [\phi].
   $$  This implies $\mu \p_\alpha f =  (\eta \mu_k) \p_\alpha f $ and it is sufficient now to prove that 
  \begin{equation}\label{for-eta-cutoff}
    \lim_{k \to \infty}    \|(\eta \mu_k) \partial_\alpha f\|_{W^{s-1, \frac sn}(\R^n)} = 0. 
\end{equation}
 In order to do so, we have to analyse the sequence $\eta \mu^k$ and its harmonic extension $(\eta \mu_k)^h$ to $\R^{n+1}_{+}$.     We have
 $$
 \|\eta \mu_k \|_{L^{\frac ns}(\R^n)}  \le \|\eta\|_{L^\infty(\R^n)} \|\mu_k\|_{L^\frac ns(B_{\rho+1})} \xrightarrow{k \to \infty} 0, 
  $$ and
  $$
  \begin{aligned}
   [\eta \mu_k ]_{W^{s, \frac ns}(\R^n)}  &  \le \|\eta\|_{L^\infty(\R^n)}  [\mu_k]_{W^{s, \frac ns}(\R^n)} 
     +  2 \ds \Big (\int_{B_{\rho +1}} |\mu_k(y)|^\frac ns \int_{\R^n} \frac{|\eta(x) - \eta(y)|^\frac ns}{|x-y|^{2n}} \,\, dx\, dy \Big ) ^{\frac sn} 
     \\ & \aleq   [\mu_k]_{W^{s, \frac ns}(\R^n)}  + \|\mu_k\|_{L^ \frac ns (B_{\rho+1})} \xrightarrow{k \to \infty} 0.
  \end{aligned}
  $$
   
 Now, following the first inequality in \eqref{wk-prod-estim}, applied  to $f$ and to the sequence $\eta\mu_k$,  together with \eqref{eq:dual2:1} and \eqref{eq:dual2:33463}, we can obtain : 
\[
\|\eta \mu_k \partial_\alpha f\|_{\dot{W}^{s-1, \frac{n}{s}}(\R^n)} \aleq  [\eta  \mu_k]_{W^{s,\frac{n}{s}}(\R^n)}\, [f]_{W^{s,\frac{n}{s}}(\R^n)} + 
\|(\eta \mu_k)^h t^{1-\frac sn -s} Df^h\|_{L^{\frac ns}(\R^{n+1}_+)}.
\]
 Since $[\eta \mu_k]_{W^{s, \frac ns}(\R^n)} \xrightarrow{k \to \infty} 0$, we conclude the proof of the theorem once we can show
\begin{equation}\label{lim=0}
\lim_{k \to \infty} \|(\eta \mu_k)^h t^{1-\frac sn -s} Df^h\|_{L^{\frac ns}(\R^{n+1}_+)}  = 0.
\end{equation} For this, we observe that 
\[
\|t^{1-\frac sn -s} Df^h\|_{L^{\frac ns}(\R^{n+1}_+)} \aleq [f]_{W^{s ,\frac ns}(\R^n)} < \infty
\]
and by the maximum principle  
\[
 \sup_{k} \|(\eta \mu_k)^h \|_{L^\infty(\R^{n+1}_+)}  \leq \sup_{k} \|\eta \mu_k \|_{L^\infty(\R^{n})}   < \infty.
\]
Let 
\[
 G_k := (\eta\mu_k)^h t^{1-\frac sn -s} |Df^h|.
\]
Then we have 
\[
 \sup_{k} |G_k(x,t)| \aleq t^{1-\frac sn -s} |Df^h(x,t)|  \quad \forall x,t \in \R^{n+1}_+.
\]
On the other hand, we have   from the convergence $\eta  \mu_k \to 0$ in $L^{\frac{n}{s}}(\R^n)$ that
every subsequence of $\eta \mu_k$    has  a subsequence $\eta \mu_{k_j} \xrightarrow{j \to \infty} 0$ almost everywhere in $\R^n$.  Since $\eta \mu_k$ are  compactly supported in $B_{\rho+1}$, they belong to $L^1(\R^n) \cap L^\infty (\R^n)$ and hence the Poisson integral formula  \eqref{poisson}  is valid. Now, the uniform boundedness of  $\eta \mu_k$  in $L^\infty(\R^n)$ and dominated convergence applied to \eqref{poisson} imply that $(\eta \mu_{k_j})^h$, and hence $G_{k_j}$, converge to 0 almost everywhere in $\R^{n+1}_+$. By dominated convergence we then find
\[
 \lim_{j \to \infty} \|G_{k_j}\|_{L^{\frac ns}(\R^{n+1}_+)} =0.
\]
A standard argument now implies \eqref{lim=0} and we can conclude the proof as \eqref{for-eta-cutoff} is shown. 
\end{proof}
The following corollary is a local version of \Cref{weak-product}: 

  \begin{corollary}\label{weak-product-local}
 Let $n \geq 2$ and $1/2<s<1$. Assume that     $\Omega \subset \R^n$ is a bounded smooth domain and $f \in W^{s, \frac ns}(\Omega)$.

{\rm (i)} Let $\mu \in W^{s, \frac{n}{s}}(\Omega) \cap L^\infty(\Omega)$. Then for any $\alpha \in \{1,\ldots,n\}$, the product $\mu \partial_\alpha f$ is well-defined as a distribution on $\Omega$ and 
$$ 
 \|\mu \partial_\alpha f\|_{\dot{W}^{s-1, \frac ns}(\Omega)}   \aleq [f]_{W^{s, \frac{n}{s}}(\Omega)} ([\mu] _{W^{s, \frac{n}{s}}(\Omega)}+ \|\mu\|_{L^\infty(\Omega)}).
$$  
Moreover, for any $\mu \in C^\infty(\overline \Omega)$ and $\phi \in W_{00}^{1-s, \frac n{n-s}}(\Omega)$ we have
\begin{equation}\label{asso-prod}
\mu \partial_\alpha f [\phi]  = \partial_\alpha f [\mu \phi].  
\end{equation}
 
{\rm (ii)} Let $\mu_k \in W^{s, \frac{n}{s}}(\Omega) \cap L^\infty(\Omega)$  be such that  
\begin{equation}\label{Lp-vanish}
 \sup_k\, \|\mu_k\|_{L^\infty(\Omega)}  < \infty \quad \mbox{and} \quad \lim_{k\to \infty} \|\mu_k\|_{W^{s, \frac{n}{s}}(\Omega)} =0. 
\end{equation} Then for any $\alpha \in \{1,\ldots,n\}$,
\[
\|\mu_k \partial_\alpha f\|_{\dot{W}^{s-1, \frac{n}{s}}(\Omega)} \xrightarrow{k \to \infty} 0.
\]
\end{corollary}
 
\begin{remark}
Note that a mere boundedness of $\|\mu_k\|_{L^\frac ns (\Omega)}$ is no more sufficient for the local version of \Cref{weak-product}-(ii) to be true. $\mu_k\equiv 1$ is a trivial counter-example. 
\end{remark}

 \begin{proof}
Given $f \in W^{s,\frac ns}(\Omega)$, $\mu \in W^{s,\frac ns}(\Omega) \cap L^\infty (\Omega)$,  we extend them to $\tilde f, \tilde \mu$ using a bounded linear operator to the whole $\R^n$ and we consider the mollified sequence $\tilde f_\e$ and $\tilde \mu_\e$. By  \Cref{weak-product}  we have for any $\phi \in C^\infty_c(\Omega)$, extended by $0$  outside $\Omega$ to $\tilde \phi$ over $\R^n$, 
$$
\int_{\Omega} \tilde \mu_\e \partial_\alpha \tilde f_\e \phi = \int_{\R^n} \tilde \mu_\e \partial_\alpha \tilde f_\e \tilde \phi \longrightarrow \tilde \mu \partial_\alpha \tilde f [\tilde \phi] \,\, \mbox{as} \,\, \e\to 0.  
$$
We define for $\phi \in C^\infty_c(\Omega)$
\begin{equation}\label{loc-prod}
 \mu \partial_\alpha  f [\phi] :=  \tilde \mu \partial_\alpha \tilde f [\tilde \phi], 
\end{equation} which, in view of the fact that 
$$
[\tilde \phi]_{W^{s,p}(\R^n)} \aleq [\phi]_{W^{s,p}_{00}(\Omega),}
$$ 
satisfies the desired estimate in (i).  Approximating $f$ and $\phi$ in their respective spaces by  smooth sequences $\tilde f_k \in C^\infty(\overline \Omega)$ and $\tilde \phi_k \in C^\infty_c(\Omega)$ and passing to the limit using the newly established  estimates on $\Omega$ yields \eqref{asso-prod}. 

As for (ii), the  \Cref{weak-product}-(ii) is applicable to the extensions $\tilde \mu_k$ because of the  assumptions in \eqref{Lp-vanish} since in that case $\|\tilde \mu_k\|_{L^\infty(\R^n)}$ are uniformly bounded and we have
 $$
\|\tilde \mu_k\|_{W^{s, \frac ns}(\R^n)} \aleq \|\mu_k\|_{W^{s, \frac ns}(\Omega)} \xrightarrow{k\to 
\infty} 0.
$$ This impies  (ii)  as formulated.

Note that a diagonal argument and part (ii) also prove the  independence of the definition from the choice of extensions.   
 \end{proof}

  \begin{corollary}\label{weak-product-neg}
Let $n \geq 2$ and $\Omega \subset \R^n$ be a bounded smooth domain or $\Omega = \R^n$.  Assume that $1/2<s<1$,  $g \in \dot{W}^{s-1, \frac ns}(\Omega), \mu \in {W}^{s, \frac{n}{s}}(\Omega) \cap L^\infty(\Omega)$,  Then the product $\mu  g$ is well-defined as a distribution on $\Omega$ and 
\[
\|\mu  g\|_{\dot{W}^{s-1, \frac{n}{s}}(\Omega)}   \aleq   \|g\|_{\dot{W}^{s-1, \frac ns}(\Omega)} ( [\mu] _{W^{s,\frac ns}(\Omega)} +  \|\mu\| _{L^\infty(\Omega)}).
\] Moreover, if $\mu_k \in W^{s, \frac{n}{s}}(\Omega) \cap L^\infty(\Omega)$ with 
\[
 \sup_k \|\mu_k\|_{L^\infty(\Omega)}   < \infty \quad \mbox{and} \quad  
 \left \{ \begin{array}{ll}\ds \lim_{k\to \infty}
 \|\mu_k\|_{W^{s,\frac{n}{s}}(\Omega)} = 0 &  \mbox{if} \,\, \Omega \neq \R^n 
 \\    \ds \sup_k \|\mu_k\|_{L^\frac ns(\R^n)}   < \infty    \,\,  \mbox{and} \,\, \lim_{k\to \infty}  [\mu_k]_{W^{s,\frac{n}{s}}(\R^n)} =0 & \mbox{otherwise}, 
 \end{array} \right.
\]  then  
\[
\|\mu_k g\|_{\dot{W}^{s-1, \frac{n}{s}}(\Omega)} \xrightarrow{k \to \infty} 0.
\]
\end{corollary}
\begin{remark}
When $\Omega = \R^n$, $1/2<s<1$ and $g$ belongs to the larger space $W^{s-1, \frac ns}(\R^n)$
the product $\mu g$ can be defined  as an element of $W^{s-1, \frac ns}(\R^n)$  and its continuity shown based on  \cite[Theorem 4.6.2/2]{RuSi},  where the Triebel-Lizorkin theory of spaces and the notion of paraproducts are used.  Another proof  can be given through duality based on \cite[Lemma 6]{BM01}. Indeed, for $1/2<s<1$, let $1 <t= n/s < \infty$, $0<\theta= (1-s)/s <1$, $1<p=n/(n-s) <\infty$, and $1<r= n/(n-1)<\infty$, and   note that 
$$
\frac 1r + \frac \theta t = \frac 1p. 
$$ Hence, \cite[Lemma 6]{BM01} implies that for all $\phi \in W^{1-s, \frac {n}{n-s}}(\R^n)$:
$$
\begin{aligned}
\|\mu \phi\|_{W^{1-s, \frac {n}{n-s}}(\R^n)} & \aleq  \|\mu\|_{L^\infty(\R^n)} \|\phi\|_{W^{1-s, \frac {n}{n-s}}(\R^n)} 
+ \|\mu\|^\theta_{W^{s, \frac {n}{s}}(\R^n)} \|\mu\|^{1-\theta}_{L^\infty(\R^n)} \|\phi\|_{L^{\frac {n}{n-1}}(\R^n)} 
\\ &  \aleq  \|\mu\|^\frac{2s-1}{s}_{L^\infty(\R^n)} \Big (  \|\mu\|^\frac{1-s}{s}_{L^\infty(\R^n)} 
+ \|\mu\|^\frac{1-s}{s}_{W^{s, \frac {n}{s}}(\R^n)}\Big ) \|\phi\|_{W^{1-s, \frac {n}{n-s}}(\R^n)}. 
\end{aligned} 
$$ Now it is sufficient to define for $g\in  W^{s-1, \frac ns}(\R^n)$:
$$ 
\mu g [\phi] := g[\mu \phi], 
$$ and we obtain the estimate
$$
\|\mu g\|_{W^{s-1,\frac ns}(\R^n)} \aleq   \|\mu\|^\frac{2s-1}{s}_{L^\infty(\R^n)} \Big (  \|\mu\|^\frac{1-s}{s}_{L^\infty(\R^n)} 
+ \|\mu\|^\frac{1-s}{s}_{W^{s, \frac {n}{s}}(\R^n)}\Big )    \|g\|_{W^{s-1, \frac ns}(\R^n)}
$$
by duality. 
\end{remark} 
\begin{proof}
 If $\Omega= \R^n$, in view of  \eqref{grad-lap}, it suffices to apply \Cref{weak-product}  to components of $f:= D\lap^{-1} g$, if necessary by approximating $g$ in $\dot{W}^{s-1, \frac ns}(\R^n)$ by a sequence of $C^\infty_c(\R^n)$ functions. If $\Omega$ is a bounded smooth domain,  we fix an extension operator $g\to \tilde g$  from $W^{s-1,\frac ns}(\Omega)$ into $W^{s-1, p}(\R^n)$, and an $\eta \in C^\infty_c(\R^n)$ such that $\eta \equiv 1$ on $\overline \Omega$.   
We have
$$
\| \eta \tilde g\|_{\dot{W}^{s-1, \frac ns}(\R^n)} \aleq \|  \tilde g\|_{{W}^{s-1, \frac ns}(\R^n)}  \aleq  \|g\|_{{W}^{s-1, \frac ns}(\Omega)} \le \|g\|_{\dot{W}^{s-1, \frac ns}(\Omega)}. 
 $$ Hence $\eta \tilde g \in  \dot{W}^{s-1, \frac ns}(\R^n)$ is a bounded extension of $g$ to the whole $\R^n$ and for any extension 
 $\tilde \mu \in W^{s, \frac ns} \cap L^\infty(\R^n)$ of $\mu$, the product $\tilde \mu (\eta \tilde g)$ is well-defined.  We let $\mu g[\phi]:= (\tilde \mu (\eta \tilde g)) [\phi]$ for all $\phi \in C^\infty_c(\Omega)$. We can now argue as in the proof of \Cref{weak-product-local} in order to establish the properties of the distributional product $\mu g$ and its independence from the choice of the extension operators or $\eta$. \end{proof}
   
\section{A Proof of Theorem~\ref{th-main1}}\label{f-develop}

Our  reasoning for proving  \Cref{prop: thin}  is a combination of  the arguments used in the proofs of \cite[Proposition 1.1]{P04} and \cite[Theorem 1.3]{LP17}.   First, analogous to \cite[Proposition~7.1] {LP17}, we show that given the proper fractional Sobolev regularity, the degree formula is valid for $f$:
\begin{lemma}\label{lem: degree}
Assume $\Omega\subset \R^2$ is an open smooth bounded set, or $\Omega = \R^2$, $s\ge 2/3$ and $f\in W^{2, 2/s} \cap C^0 (\Omega, \R^2)$. For any $\TO \Subset \Omega$ and any $g \in C^\infty_c( \R^2 \setminus f(\p\TO))$, one has
\begin{equation*}
\int_{\R^2} g(y) \deg(f, \TO; y)\,\dd y =  \int_\TO {\rm Jac}(f) [g\circ f].
\end{equation*}
In particular, if ${\rm Jac(f)}>0$, then  $\deg(f, \TO; y)$ is nonnegative whenever it is well-defined and moreover: 
\begin{equation}\label{deg>0}
\forall y \in f (\TO) \setminus f(\partial \TO)\quad \deg(f, \TO; y) \ge 1,
\end{equation} since the degree must be positive for such $y$. 
 
\end{lemma}
By definition ${\rm Jac}(f)>0$ if  for all non-negative $\phi\in C^\infty_c(\Omega)$, ${\rm Jac} (f)[\phi]>0$, unless $\phi \equiv 0$. 

\begin{proof}
	Consider the mollified functions $f_\e := f \ast \varphi_\e \in C^\infty(\Omega,\R^2)$,  as defined in \Cref{mollif}. Since $f_\e$ converges locally uniformly to $f$,  similar as in \cite[Proposition~7.1]{LP17}  we have
	\begin{eqnarray*}
	&&\deg(f, \TO; y)=\deg(f_\e, \TO; y) \qquad \text{ for all } y \in {\rm supp}\,g;\\
	&&
	\int_\TO \big(g \circ   f_\e(z)\big)\det\na f_\e (z)\,\dd z = \int_{\R^2} g(y)\deg(f_\e, \TO; y)\,\dd y,
	\end{eqnarray*}
for small enough $\e$. So it suffices to show that
\begin{align}\label{degree, formula 1}
\int_\TO \big(g \circ f_\e(z)\big)\det\na f_\e (z)\,\dd z \longrightarrow {\rm Jac}(f) [g\circ f] \qquad\text{as $\e \to 0$}.
\end{align}
But the left-hand side of Eq.~\eqref{degree, formula 1} equals ${\rm Jac}(f_\e) [g\circ f_\e]$, which converges to ${\rm Jac}(f) [g\circ f]$ by \Cref{la:distjac}.  This proves Eq.~\eqref{degree, formula 1}, and hence the assertion follows.  \end{proof}

Next we show that if further ${\rm Jac}(f)\equiv 0$ and ${\rm curl} \, f =0$, then  the image $f(\Omega)$ is  of zero measure.  In view of \cite[Corollary 1.1.2]{kor2} and \cite[Proposition 2.1]{DLP20}, it follows that $f$ is either locally constant  around a point or constant in segments joining the boundary of $\Omega$ on both sides. The local  H\"older regularity $C^{0,s/2}$ is a straightforward  consequence of the  Fubini theorem  for fractional Sobolev spaces \cite[2.3.4/2]{RuSi} and the Sobolev embedding Theorem  in one dimensions \cite[Theorem 8.2]{DNPV12} after the application of the local bilipschitz change of variable introduced in \cite[Lemma 2.11]{DLP20}. The little H\"older regularity follows in view of density of smooth mappings in $W^{s, \frac 2s}(\R)$ for $s>0$.  This will conclude the proof of \Cref{th-main1}.

\begin{proposition}\label{prop: thin}
Let $\Omega$, $s$, and $f$ be as in the assumptions of \Cref{th-main1}. Then $f(\Omega)$ has zero Lebesgue measure. In particular it has empty interior.  
\end{proposition}

\begin{proof}
Without loss of generality and by considering compactly contained subsets of $\Omega$ we can assume that $\Omega$ is bounded and smooth. Following Kirchheim \cite{K01} and as in the arguments Pakzad \cite[Lemma 2.1]{P04} and Li--Schikorra \cite[Theorem~1.6]{LS19}, consider the auxiliary maps
\begin{equation}\label{f delta}
f^{(\delta)}(x,y):=f(x,y)+\delta (-y,x)^\top.
\end{equation}  Let $\TO \Subset \Omega$ be an open set.   
Since $f^{(\delta)} \to f$ uniformly as $\delta \searrow 0$, there exists  a number $\delta_\kappa$ small enough such that 
$$\|f-f^{(\delta_\kappa)}\|_{C^0(\TO)} \leq \kappa.
$$ One may choose $\delta_\kappa$ to be  decreasing in $\kappa$.    As a consequence, $f(\TO)$ lies in the $\kappa$-neighbourhood of $f^{(\delta_\kappa)}(\TO)$. Thus
\begin{align*}
\leb^2\Big( f(\TO) \Delta f^{(\delta_\kappa)}(\TO) \Big) \leq C\kappa^2
\end{align*}
for some constant $C$ depending only on $\TO$. Therefore, by sending $\kappa \map 0$, we may infer that 
\begin{equation}\label{delta convergence}
\lim_{\delta \searrow 0} \leb^2\big(f^{(\delta)}(\TO)\big) = \leb^2\big( f(\TO) \big).
\end{equation}
On the other hand, once again by setting $f_\e := f \ast \varphi_\e$ and $f_{\e}^{(\delta)}(x,y):= f_\e(x,y)+  \delta (-y,x)^\top $, we note that $f^{(\delta)}$ is the $W^{s,2/s}$-limit of $f_{\e}^{(\delta)}$ and hence  for all $\phi \in C^\infty_c(\Omega)$:
 $$
{\rm Jac} (f^{(\delta)}) [\phi] = \lim_{\e\to 0} \int_\Omega \det(\nabla f^{(\delta)}_\e) \phi   = \lim_{\e\to 0} \int_\Omega \det(\nabla f_\e )\phi + \int_\Omega  \delta^2 \phi =  \int_\Omega  \delta^2 \phi,
$$  where we used the facts that  ${\rm curl}\, f_\e=0$ and ${\rm Jac} (f)=0$.  We deduce that ${\rm Jac}(f^{(\delta)}) \equiv \delta^2>0$. 
Note that by \cite[Theorem 1.1]{LS19}  $f^{(\delta)}$ is continuous.
 
We take a  nondecreasing sequence  of   nonnegative $g_k \in C^\infty_c (\R^2\setminus f^{(\delta)}(\partial \TO))$ converging pointwise to  
$\chi_{\R^2 \setminus f^{(\delta)}(\partial \TO)}$. Applying  \Cref{lem: degree} and  the monotone convergence  theorem we have
\begin{equation}\label{deg: L1estiamte}
\begin{aligned}
\int_{\R^2\setminus f^{(\delta)}(\partial \TO)} \deg(f^{(\delta)}, \TO; y) \, dy & = \lim_{k\to \infty} \int_{\R^2} g_k (y) \deg(f^{(\delta)}, \TO; y) \,dy 
= \lim_{k\to \infty}\int_\TO (g_k \circ f^{(\delta)}) \delta ^2 \\ & = \delta^2 \leb^2\big( \TO\setminus  (f^{(\delta)})^{-1}(f^{(\delta)}(\partial \TO))\big) \le \delta^2 \leb^2\big( \TO\big).
\end{aligned}
\end{equation}  For any $x\in \Omega$, we let $B_x \Subset \Omega$ be a disk centered at $x$ in a manner that $f^{(\delta)}\in W^{s,2/s}(\partial B_x)$. This is possible by the Fubini theorem for fractional Sobolev spaces, which is a well-known fact, see \cite{St68} in view of  \cite[Lemma 7.68]{adams}.  A similar proof  recently appeared in \cite[Lemma 2.2]{LS19}; for other proofs see \cite[Theorem 2.3.4/2]{RuSi} or \cite[Lemma 2.6]{SS20}.   Now, \Cref{th-hausdorffcontentzero}  yields $\leb^2\big(f^{(\delta)} (\partial B_x)\big )=0$.  Therefore, applying \eqref{deg>0} and \eqref{deg: L1estiamte} to $\TO = B_x$ we have: 
$$
\begin{aligned}
\leb^2\big(f^{(\delta)}(B_x)\big) & = \int_{\R^2}   \chi_{f^{(\delta)}(B_x)} = \int_{\R^2\setminus f^{(\delta)}(\partial B_x)}   \chi_{f^{(\delta)}(B_x)} \\ & \le \int_{\R^2\setminus f^{(\delta)}(\partial B_x)} \deg(f^{(\delta)}, B_x; y)  dy \le \delta^2 \leb^2\big(B_x \big). 
\end{aligned}
$$  It follows by \eqref{delta convergence} that for all $x\in \Omega$, $\leb^2\big(f(B_x) \big)=0$. The conclusion follows. \end{proof}

\section{Mollifying \texorpdfstring{$W^{1+s, 2/s}$}{W(1+s,2/s)} isometric immersions}\label{sec-mol}

Given an isometric immersion  $u\in I^{1+s, \frac2s} (\Omega, \R^3)$ on a bounded smooth domain $\Omega\subset \R^2$ , with $s>1/2$,  We will study the geometry of a sequence of mollified mappings $u_\e := u\ast \varphi_\e$. One difficulty is that the mapping $u_\e$ is not isometric anymore, and  a priori might fail to be an immersion. We will also need to define the Gauss map $\vec n^\e$ by the formula
$$
\vec n^\e := \frac{\partial_1 u_\e \wedge \partial _2 u_\e}{|\partial_1 u_\e \wedge \partial _2 u_\e|}.  
$$ But $\vec n^\e$ is well-defined only if $|\partial_1 u_\e \wedge \partial _2 u_\e|(x) \neq 0$ for a.e.\@ $x \in \TO$. Actually, for $\vec n^\e$ to be smooth we need a uniform lower bound on  $|\partial_1 u_\e \wedge \partial _2 u_\e|$; in other words we need that $u_\e$ is an immersion at least for small enough $\e>0$, a fact that is true but by no means trivial.  This is the subject of the following lemma, which also discusses the behavior of the pull-back metric induced by $u_\e$, i.e.\@:
$$
\g^\e:= (\nabla u_\e)^T \nabla u_\e = u_\e^\ast \eucl.
$$

 \begin{lemma}\label{iso-mol-immersion}
Let $\Omega \subset \R^2$ be a bounded smooth domain, $0<s<1$ and $u\in  I^{1+s, \frac2s} (\Omega, \R^3)$.  Let  $\TO \Subset \Omega$. Then  there exists $\e_0>0$ such that 
for all $0<\e<\e_0$, 
\begin{equation}
\label{below-bound}
\forall x\in \TO \quad  |\partial_1 u_\e \wedge \partial _2 u_\e|(x) = \sqrt{\det \g^\e} >\frac 12
\end{equation} and  as a consequence $\g^\e$ is a Riemannian metric,  $u_\e: \TO \to \R^3$ is a  smooth immersion on $\TO$ and the unit normal $\vec n^\e$ and the second fundamental form  $\two^\e_{ij} := \partial_{ij} u_\e \cdot \vec n^\e$ are well-defined. 
Moreover,  the following statements hold true
\begin{itemize}
\smallskip
\item[\rm (i)] 
$
\ds \lim_{\e \to 0} \|\g^\e - \eucl\|_{C^0(\TO)} =0. 
$ 
\smallskip
\item[\rm (ii)] 
$
\ds \lim_{\e \to 0} \|(\g^\e)^{-1} - \eucl\|_{C^0(\TO)}=0. 
$
\smallskip
\item[\rm (iii)]  $\|\g^\e - \eucl\|_{L^\frac 2s (\TO)} \le o(\e^{s})$  and $\|\nabla \g^\e\|_{L^{\frac 2s}(\TO)} + \|\nabla (\g^\e)^{-1}\|_{L^{\frac 2s}(\TO)} \le o(\e^{s-1})$. 
\smallskip
\item[\rm (iv)]  $\|\g^\e - \eucl\|_{L^\frac 1s(\TO)} \le o(\e^{2s})$ and for $k\ge 1$, $\|\nabla^k \g^\e\|_{L^{\frac 1s}(\TO)} \le o(\e^{2s-k})$. 
\smallskip 
\item[\rm (iv)] $\ds \lim_{\e \to 0} \|\g^\e - \eucl\|_{W^{2s, \frac 1s}(\TO)} =0$. 
 \end{itemize}
 \end{lemma}
\begin{remark}
$W^{s,\frac 2s}$ barely fails to embed in $L^\infty$ in two dimensions and  the $C^0$ convergence of the metrics $\g^\e$, which is a key feature 
of the statement, is not trivial.
\end{remark}
 
 \begin{proof}
 Consider the smooth manifold
 $$
 O(2,3):= \{ A\in \R^{3\times 2} : \quad A^T A = {\rm Id}\}, 
 $$ and note that if $u\in I^{1+ s, \frac 2s}(\Omega)$, then $\nabla u \in O(2,3)$ a.e.\@ in $\Omega$.  We claim that the Jacobian derivatives  $\nabla u_\e$ of the mollified sequence $u_\e$ are uniformly close to $O(2,3)$ on $\TO$. Note that $W^{s, \frac 2s}(\TO)$ does not embed in $L^\infty(\TO)$ and so $\nabla u_\e$ are not necessarily uniformly close to $\nabla u$. 

Lacking an $L^\infty$ estimate, the main idea is to use the approach of Schoen and Uhlenbeck \cite{SchU83}  and to apply the standard BMO estimate 
$$
\|\nabla u_\e - \nabla u\|_{BMO} \le \|\nabla u - \nabla u\|_{W^{s, \frac 2s}},
$$ on small balls around a point $x\in \Omega$.  See also \cite[Section I.1]{BNi95} for a discussion of this  topic and its applications in a larger context and \cite{BPSch13} regarding its application in approximating fractional Sobolev mappings into manifolds. 

Indeed,  applying \Cref{VMO} we have for all $x\in \TO$ and $\e< {\rm dist} (\TO, \partial \Omega)$:
$$
|{\rm dist} (\nabla u_\e(x), O(2,3))|^\frac ns \le \fint_{B_\e(x)} |\nabla u(y) - \nabla u_\e(x)|^\frac ns  \, dy \le o(1),    
$$ 
where 
$$
{\rm dist} (\nabla u_\e(x), O(2,3)):= \inf_{A\in O(2,3)} |\nabla u_\e(x)- A|,
$$ and $|A|$ denotes the Hilbert-Schmidt norm of a matrix $A$.  Let $A(x) \in O(2,3)$ be the matrix for which the infimum is attained.  
Therefore
$$
\begin{aligned}
\|\g^\e - \eucl\|_{C^0(\TO)} & = \sup_{x\in \TO}   |(\nabla u_\e (x))^T \nabla u_\e (x) - {\rm Id} |  = \sup_{x\in \TO}   |(\nabla u_\e)^T \nabla u_\e   -  A^TA |(x) \le o(1),
\end{aligned}
$$ which proves (i). In particular,  since $|\partial_1 u_\e \wedge \partial_2 u_\e| = \sqrt{\det \g^\e}$ we also obtain
$$
\lim_{\e\to 0}\||\partial_1 u_\e \wedge \partial_2 u_\e| - 1\|_{C^0(\TO)} = \lim_{\e\to 0} \|\sqrt{\det \g^\e} -1 \|_{C^0(\TO)} =0.
$$ This establishes   \eqref{below-bound}. The statements (ii) follows by straightforward calculations using the above uniform estimates. Since $\nabla u_\e$ stays uniformly bounded in $L^\infty(\TO)$, applying \Cref{moli-est}(ii) to the sequence $\nabla^2 u_\e$ yields (iii). Finally (iv)  and (v) follow respectively from the commutator estimate  \Cref{moli-est}(iii)  and  \Cref{commute-convergence} since  
$$
((\nabla u)^T \nabla u) \ast \varphi_\e =\ \eucl \ast \varphi_\e = \eucl \,\,\, \mbox{in} \,\,\, \TO
$$ for all $\e < {\rm dist (\TO, \partial \Omega)}$.
\end{proof}

We can therefore define the second fundamental form of $\g^\e$ on $\TO$ by
\begin{equation}\label{iso-mol-2nd}
\two^\e_{ij} := \p_{ij} u_\e \cdot \vec n^\e.
\end{equation} Also, remember that for any Riemannian metric $\g \in R^{2\times 2}_{{\rm sym}+}$, its Christoffel symbols are defined by
$$
\Gamma^l_{ij}(\g):= \frac 	12  \g^{lm} ( \partial_{i} \g_{mj}+ \partial_j \g_{im} - \partial_{m} \g_{ij}),  
$$ with the Einstein summation convention, where $\g^{lm}$ are the components of $\g^{-1}$.   We define therefore the tensor $\Gamma^\e$ by:
$$
\Gamma^\e := [\Gamma^{l,\e}_{ij}]_{i,j,l\in \{1,2\}},\quad \G^{l,\e}_{ij}:= \G^l_{ij} (\g^\e),
$$    with the usual convension $\ds |\Gamma^\e|:= (\sum_{i,j,l=1}^2 |\Gamma^{l,\e}_{ij}|^2) ^\frac 12$. 

{\begin{corollary}\label{Christoffel}
Let $\Omega, \TO, s$ and $u$ be as in \Cref{iso-mol-immersion}. Then  
\begin{itemize}
\smallskip
\item[{\rm (i)}]   $\|\two^\e\|_{L^\frac 2s (\TO)}  \le o(\e^{s-1})$ and $\|\two^\e\|_{\dot{W}^{s-1, \frac 2s}(\TO)} \le C$. 
\smallskip
\item[\rm (ii)]  $\|\Gamma^\e\|_{L^\frac 1s (\TO)} \le o(\e^{2s-1})$ and $\|\nabla \Gamma^\e\|_{L^\frac 1s (\TO)} \le o(\e^{2s-2})$. 
\smallskip
\item[\rm (iii)] If $s\ge \frac 12$, then $\ds \lim_{\e\to 0}  \|\Gamma^\e\|_{W^{2s-1, \frac 1s}(\TO)} =0$. 
\end{itemize}
\end{corollary}
}

\begin{proof}
(i) follows from \Cref{moli-est}(ii) applied to $\nabla^2 u_\e$ and  from \Cref{weak-product-local} with the estimate
$$
\|\two^\e\|_{\dot{W}^{s-1, \frac 2s}(\TO)} \le (\|\vec n_\e\|_{L^\infty(\TO)} + [\vec n_\e]_{W^{s,\frac 2s}(\TO)})  [\nabla u_\e]_{W^{s, \frac 2s}(\TO)} \le C,
$$ where the uniform bounds on $\vec n_\e$ are obvious from $\eqref{below-bound}$ and the similar bounds on $\nabla u_\e$. Applying  \Cref{iso-mol-immersion}(iii)-(iv) we obtain (ii) on $\TO$:  
   $$
\| \Gamma^\e\|_{L^{\frac 1s}} \le \|(\g^\e)^{-1}\|_{L^\infty} \|\nabla \g\|_{L^\frac 1s} \le o(\e^{2s-1}),
 $$  and
  $$
 \begin{aligned}
 \| \nabla \Gamma^\e\|_{L^{\frac 1s}} &  \le \|\nabla (\g^\e)^{-1}\|_{L^\frac 2s} \|\nabla \g^\e\|_{L^\frac 2s}  + \|(\g^\e)^{-1}\|_{L^\infty} \|\nabla^2 \g^\e\|_{L^\frac 1s}  \le o(\e^{2s-2}). 
 \end{aligned}
 $$ Interpolating these two estimates similar as in  \Cref{commute-convergence} yields (iii). 
\end{proof}
 
 Our next statements regard the asymptotic behavior of $\det \two^\e$, which enjoys a better than expected convergence due to its almost Jacobian determinant structure, and of ${\rm curl} \, \two^\e$:

\begin{proposition}\label{det2nd-e} 
Let $\Omega, \TO, u$ be as in \Cref{iso-mol-immersion} with $s\ge 1/2$. Then  for all $\phi \in C^\infty_c(\TO)$
$$
\Big |\int_\TO (\det \two^\e) \phi \Big | \le o(\e^{2s-1}) \|\nabla \phi\|_{L^{\frac{1}{1-s}}(\TO)} + o(1)  \|\phi\|_{L^\infty(\TO)}.
$$
\end{proposition}

 \begin{proof}
   By  \cite[Equations (2.1.2)]{hanhong} and the Gauss equation \cite[Equations (2.1.7)]{hanhong} we have on $\TO$:
$$
\begin{aligned}
\det \two^\e & = R_{2121}(\g^\e) = \g^\e_{1m}(\partial_1 \Gamma^{m,\e}_{22} - \partial_2 \Gamma^{m,\e}_{21} + \Gamma^{m,\e}_{1s} \Gamma^{s,\e}_{22}  
- \Gamma^{m,\e}_{2s} \Gamma^{s,\e}_{21})\\ & =
\partial_1 (\g^\e_{1m} \Gamma^{m,\e}_{22} ) - \partial_2 ( \g^\e_{1m} \Gamma^{m,\e}_{21})  +  O (|\Gamma^\e|^2) \\ &  =  2 \partial_{12} \g^\e_{12} - \partial_{11} \g^\e_{22} - \partial_{22} \g^\e_{11} +  O (|\Gamma^\e|^2)  \\ & = - {\rm curl}^T {\rm curl}\,\, \g^\e  + O (|\Gamma^\e|^2).
   \end{aligned} 
$$  Hence, using the embedding of $W^{2s-1, \frac 1s}(\R^2)$ into $L^2(\R^2)$ and \Cref{Christoffel}(iii) for $s\ge \frac 12$ we obtain:
$$
\begin{aligned}
\Big | \int_\TO (\det \two^\e)\phi \Big | 
& 
\aleq \int_\TO \Big | ({\rm curl}^T {\rm curl}\,\, \g^\e) \phi \Big |  + \|\Gamma^\e\|^2_{L^2(\TO)} \|\phi\|_{L^\infty(\TO)} 
\\ &
\aleq  \Big |  \int_\TO  ({\rm curl}\, \g^\e)  \cdot \nabla^\perp \phi \Big | + \|\G^\e\|^2_{W^{2s-1, \frac{1}{s}}(\TO)} \|\phi\|_{L^\infty(\TO)} 
\\ & 
\aleq \|\nabla \g^\e\|_{L^\frac 1s(\TO)} \|\nabla \phi\|_{L^\frac{1}{1-s}(\TO)} + o(1)    \|\phi\|_{L^\infty(\TO)}, 
\end{aligned} 
$$ which concludes the proof in view of \Cref{iso-mol-immersion}(iii). 
\end{proof}

\begin{proposition}\label{Curl2nd-e}
Let $\Omega, \TO, u$ be as in \Cref{iso-mol-immersion} with $\frac 23 \leq s < 1$.  
\begin{equation}\label{curlof2nd}
 \|\p_2 \two_{i1} ^\e - \p_1 \two^\e_{i2} \|_{L^1 (\TO)} \le o(1) \quad \mbox{for}\,\, i=1,2.
\end{equation} 
\end{proposition}
\begin{remark} \label{L1-not-enough}
An $L^1$ estimate for ${\rm curl}\, \two^\e$  is not enough for a Hodge decomposition for $\two^\e$, hence a better than $L^1$ estimate is crucial for  completing the same proof as in \cite{DLP20} for our main theorem.   We will hence adapt  a new approach as explained in the following section.
\end{remark}

\begin{proof}
The Codazzi-Mainardi equations \cite[Equation (2.1.6)]{hanhong} for the immersion $u_\e$  read
$$
\partial_2 \two^{\e}_{i1} -\partial_1 \two^{\e}_{i2}= \two^{\e}_{l1} \Gamma_{i2}^{l, \e}
 -\two^{\e}_{l2} \Gamma_{i1}^{l,\e}. 
$$  Now since $s \ge \frac 23$,  $ \frac 12 < s':= \frac{2-s}{2} \le s$,  the embedding
$$
W^{s, \frac 2s}(\R^2)   \hookrightarrow W^{s', \frac 2{s'}}
$$ implies that $u \in I^{1+s', \frac 2{s'}}(\Omega, \R^3)$. Applying \Cref{Christoffel}(i) for $s$  and \Cref{Christoffel}(ii) for $s'$ yields
for any $\phi\in C^\infty_c(\Omega)$,  
\[
\|\partial_2 \two^{\e}_{i1} -\partial_1 \two^{\e}_{i2}\|_{L^1(\TO)} \aleq \|\two^\e \|_{L^{\frac 2s}(\TO)} \| \G^\e\|_{L^\frac{1}{s'}(\TO)} \le o(\e^{s-1}) o(\e^{2s'-1}) = o(1). \qedhere
\]  \end{proof}
   
\section{Second fundamental form for \texorpdfstring{$W^{1+s, 2/s}$}{W(1+2,2/s)}
isometric immersions for \texorpdfstring{$\frac 12<s < 1$}{1/2<s<1}}\label{2nd}

Given an isometric immersion  $u\in I^{1+s, \frac2s}_{loc}(\Omega, \R^3)$ and a bounded smooth domain $\TO \Subset \Omega$, with $s>1/2$, we shall define a weak notion of the second fundamental form $\two$  as a distribution in $\dot{W}^{s-1,\frac{2}{s}} (\TO, \gl)$.  In order to apply the results of \Cref{sec-mol},  note that for a $\delta>0$ small enough $\TO\Subset \TO_\delta \Subset \TO_{2\delta}  \Subset \Omega$ and $u\in I^{1+s, \frac2s}(\TO_{2\delta}, \R^3)$, where $
\TO_\delta:= \{x\in \R^2:\,\, {\rm dist}(x, \TO)<\delta\}.$

The second fundamental form  of a given immersion $u:\TO \to \R^3$ in the chart defined by $u$ itself is expressed by the  product 
\begin{equation}\label{2ndform}
\RN{2} _{ij}:= \partial_{ij} u \cdot \vec n,
\end{equation} where for all $x\in \TO$, $\vec n(x)$ is the unit normal to the immersed surface $u(\TO$) at $u(x)$.  Under our regularity assumptions,  and since by the isometry condition  for a.e.\@ $x\in \TO$ we have  $|\partial_i  u(x)| = 1$, we have $\nabla u\in L^\infty(\TO, \R^{3\times 2})$, and so the unit normal 
\begin{equation}\label{unit-normal}
\vec n := \partial_1 u \wedge \partial _2 u \in W^{s,\frac 2s} \cap L^\infty (\TO).
\end{equation}  On the other hand, $\partial_{ij} u \in W^{s-1, \frac 2s}_{loc}(\TO, \R^3)$, and could  be a mere distribution. However, the existence of  the distributional product $\RN{2} _{ij}$ under these regularity assumptions  is  justified by \Cref{weak-product}. To summarize we state the following definition:  

\begin{definition}\label{def: II}
Let $\Omega \subset\R^2$ be an open set and $u\in I^{1+s, \frac2s}_{loc}(\Omega, \R^3)$ with $\frac 12<s<1$. Then, through  \Cref{weak-product}, we may define its (weak) second fundamental form  
\begin{equation*}
\two= \two (u):= [\two_{ij}] _{i,j\in\{1,2\}} \in W^{s-1, \frac{2}{s}}_{loc} \big(\Omega,\gl \big), 
\end{equation*} by \eqref{2ndform}, namely, 
$$
\two_{ij}[\phi] := \sum_{k=1}^3 (\vec n^k \partial_{ij}u^k )  [ \phi]
$$ for all $\phi \in W^{1-s, \frac 2{2-s}}_0(\Omega)$ with $\supp \phi\Subset \Omega$.
\end{definition}

 \begin{proposition}\label{2nd-convergence}
 Let $\Omega \subset \R^2$ be an open set, $\frac 23\le s<1$ and $u\in  I^{1+s, \frac2s}_{loc} (\Omega, \R^3)$. For all bounded smooth domain $\TO \Subset \Omega$  
  \begin{itemize}
 \smallskip
\item[{\rm (i)}]$ \ds \lim_{\e \to 0} \|\vec n^\e - \vec n\|_{W^{s, \frac 2s}(\TO)} =0$.  
  \smallskip
\item[{\rm (ii)}] $\ds  \lim_{\e \to 0} \|\two^\e - \two\|_{\dot{W}^{s-1, \frac 2s}(\TO)} =0$. 
  \end{itemize}
 \end{proposition}

\begin{proof}
Note that  for any subsequence, we can always find a subsequence of $\nabla u_\e$ converging point-wise to  
$\nabla u$ and that $\nabla u_\e$ are uniformly $L^\infty$-bounded in $\e$.  Hence, a straightforward norm calculation and dominated convergence theorem implies (i). To show (ii), we write
$$
\two^\e_{ij} - \two_{ij} = \sum_{k=1}^3 \vec n^{\e, k}\partial_{ij} {u^k_\e} - \vec n^{k}\partial_{ij} {u^k} =  \sum_{k=1}^3  \vec n^{\e, k}(\partial_{ij} {u^k_\e} -  \partial_{ij} {u^k}) +  \sum_{k=1}^3 (\vec n^{\e, k}  - \vec n^{k}) \partial_{ij} {u^k},
$$ where $\vec n^\e = (\vec n^{\e, 1}, \vec n^{\e, 2}, \vec n^{\e, 3})$. Now in view of (i), the convergence of each summing term in the $\dot{W}^{s-1, \frac 2s}(\TO)$ norm follows   in order from  the first and second parts of  \Cref{weak-product-local}. 
\end{proof}
 An immediate conclusion of \Cref{Curl2nd-e}  is the following statement regarding the second fundamental form of $u$ when $s\ge \frac 23$: 

\begin{lemma}\label{two=nab-f}
Let $\frac 23 \le s <1$  and $u \in I_{loc}^{1+ s, \frac 2s}(\Omega, \R^3)$. Assume that $\TO \Subset \Omega$  is a simply connected bounded smooth domain and let $\two$ be as in Definition~\ref{def: II}. Then there exists $f \in W^{s,\frac{2}{s}}(\TO,\R^2)$ such that $\two = \na f$ in the sense of distributions. \end{lemma}

\begin{proof}
An immediate consequence of \Cref{Curl2nd-e} is that  ${\rm curl} \, \two$ satisfies the Codazzi equations in the sense of distributions, i.e.\@
 ${\rm curl} \, \two=0$:
\begin{equation}\label{codazzi} 
\p_2 \two_{11} -  \p_1 \two_{12}   =0 \quad \text{ and } \quad  \p_2 \two_{21} - \p_1 \two_{22} =0\quad \mbox{in} \,\, \mathcal D'(\TO_\delta).
\end{equation}
 Let us  consider a direct regularisation of the second fundamental form $\two$. With $\two$ defined as in \Cref{def: II}, we set
\begin{equation}\label{two-mol}
\two_\e:= \two\ast \varphi_\e \in C^\infty\big(\TO_\delta ;\gl\big).
\end{equation}
Here $\two_\e \map \two$ in $W^{s-1,\frac{2}{s}}(\TO_\delta)$ as $\e \map 0$. The order of convolution and differentiation can be interchanged, so $\two_\e$  satisfies  \eqref{codazzi} in $\mathcal D'(\TO)$ for $\e<\delta$. Therefore, since $\TO$ is simply-connected, there exists $f^\e \in C^\infty(\Omega,\R^2)$ such that $\two_\e=\na f^\e$. By standard elliptic regularity theory we may choose $f^\e$ to be convergent to some $f$ in $W^{s,\frac{2}{s}}$. Since $\two_\e$ converges only in a very weak norm, and we must be careful that the traces of the solutions are well-defined on the boundary, hereby we justify these estimates.

 In order to find the sequence $f^\e$,  we first solve for 
$$
\left \{ \begin{array}{ll}
\Delta {\widetilde \two}_\e = \two_\e & \mbox{in} \,\, \TO \\
{{\rm curl} \, \widetilde \two}_\e = \p_\nu \widetilde \two \cdot \tau - \p_\tau \widetilde \two \cdot \nu = 0 &  \mbox{on} \,\, \partial \TO 
\\   \widetilde \two \cdot \nu =0 &   \mbox{on} \,\, \partial \TO
\end{array} \right .
$$ where $\nu$ and $\tau := \nu^\perp$ are respectively the outward normal and tangential fields to $\partial \TO$.  Note that the above system is a  basic elliptic system discussed at length in the literature of elliptic systems for differential forms, see e.g.\@  \cite[Lemma 1.6.5]{Schw95}.  However, from another point of view, if we flatten the boundary the Dirichlet and Neumann boundary conditions decouple and so there is no problem in directly applying the theory of  elliptic equations. By \cite[Theorem 3.4.3/3(i)]{RuSi}, 
$\widetilde \two_\e$ satisfies the estimate
$$
\|{\widetilde \two}_\e\|_{W^{1+s, \frac 2s}(\TO)} \aleq \|\two_\e\|_{_{W^{s-1, \frac 2s}(\TO)}} \le C. 
$$ Taking the curl of the equation, we note that ${\rm curl}\, \widetilde \two_\e$ is harmonic and vanishes on the boundary, hence  ${\rm curl}\, \widetilde \two_\e \equiv 0$ in $\TO$. Now we use the identity
$$
\nabla {\rm div}\, \widetilde \two_\e  - \Delta \widetilde \two_\e   = - \nabla^\perp {\rm curl} \, \widetilde \two_\e,
$$ to deduce that $f^\e := {\rm div}\,\widetilde \two_\e$  satisfies $\nabla f^\e = \two_\e$ with the estimate 
$$
 \|f^\e\|_{W^{s,\frac 2s}(\TO)} \aleq \| \widetilde \two_\e\|_{W^{1+s, \frac 2s}(\TO)} \aleq 1. 
 $$  Therefore $f^\e$ converges  in the sense of distributions to some   $f \in W^{s, \frac 2s}(\TO)$  satisfying $\nabla f = \two$.  
 
 \end{proof}

\section{Developability of components and \texorpdfstring{$c^{1, \frac s2}$}{c1s/2}-regularity}\label{sec-compnt}
\begin{theorem}\label{compnt-develop}
Let  $\Omega \subset \R^2$ be a  bounded smooth domain and assume that $u\in  I^{1+s, \frac{2}{s}}_{loc}(\Omega,\R^3)$ with $\frac{2}{3}\leq s<1$. Then for each $m\in \{1,2,3\}$,   the component $u^m$ satisfies 
$$
\Jac(\nabla u^m) \equiv 0 \,\, \mbox{in}\,\, \mathcal D' (\Omega)
$$
and as a consequence is $c^{1,\frac s2}$-regular  and   developable by \Cref{th-main1}. 
\end{theorem}
 
 \begin{proof}
  The argument follows closely that of \cite[Theorem 3]{DLP20}. Let  us fix $m\in \{1,2,3\}$ and set 
 $$
 g:=  \nabla u^m \in W_{loc}^{s, \frac 2s}(\Omega, \R^2).
 $$   Let $\TO \Subset \Omega$. For $\delta>0$ small enough we have $u\in  I^{1+s, \frac{2}{s}}(\TO_\delta,\R^3)$. For $\e<\delta$ we let $u_\e$ be the mollified sequence of immersions with the properties discussed in \Cref{sec-mol}. Note that by  \cite[Equation (2.1.3)]{hanhong} we have
 \begin{equation}\label{covar}
 \partial_{ij} u_\e^m = \Gamma^{k,\e}_{ij} \partial_k u^m_\e + \two^\e_{ij} \vec n^{\e,m}.
\end{equation} 
 Obviously $g_\e = \nabla u_\e^m$  and hence   for all $\phi\in C^\infty_c(\TO)$:
 $$
\begin{aligned}
 \int_\TO \Jac (g_\e) \phi  & =    \int_\TO \det (\nabla^2 u^m_\e) \phi     
=  \int_\Omega  {\rm det} (\Gamma^{\e} \cdot \nabla u^m_\e+  \two^\e \vec n^{\e,m}  ) \phi \\ & = 
  \int _\Omega \det (\two^\e) (\vec n^{\e,m})^2 \phi  +      \int_\Omega {\rm det} (\Gamma^{\e}\cdot \nabla u^m_\e) \phi  
+     \int_{\Omega}  \vec n^{\e,m} \two^\e   : {\rm cof}(\Gamma^\e \cdot \nabla u_\e^m ) \phi \\ & = I_1^\e + I^\e _2 + I^\e_3.
\end{aligned} 
$$  We  claim that as $\e\to 0$ the limit of each term $I^\e_j$ is $0$, which will complete the proof as $\Jac (g)$ is the distributional limit of ${\rm Jac}(g_\e)$ \cite[Lemma 1.3]{LS19}.   By \Cref{det2nd-e} 
$$
\begin{aligned}
|I_1^\e| & \le o(\e^{2s-1}) \Big (\| (\nabla \vec n^{\e,m}) \phi \|_{L^\frac 1{1-s}(\TO)}  + \|\nabla \phi\|_{L^\frac 1{1-s}(\TO)}\Big ) + o(1) \| (\vec n^{\e,m}) ^2\phi\|_{L^\infty(\TO)} 
\\ &
\le  o(\e^{2s-1}) \|\nabla \vec n^{\e,m}\|_{L^\frac 1{1-s}(\TO)} \|\phi \|_{L^\infty(\TO)} + o(\e^{2s-1})  \|\nabla \phi\|_{L^\frac 1{1-s}(\TO)}  + o(1) \| \phi\|_{L^\infty(\TO)} . 
\end{aligned} 
$$ However note  for $s\ge \frac 23$ the embedding
$$
\|\nabla u_\e \|_{W^{2(1-s), \frac 1{1-s}}(\TO)} \aleq  \|\nabla u_\e \|_{W^{s, \frac 2s}(\TO)} \aleq 1. 
$$ Therefore applying \Cref{moli-est}(ii) to $\nabla u_\e$ and in view of \eqref{below-bound} we obtain
$$
\|\nabla \vec n^{\e,m}\|_{L^\frac 1{1-s}(\TO)} \le o(\e^{2(1-s) -1}) \le o(\e^{1-2s}). 
$$ We conclude for $I_1^\e$ that
$$
|I_1^\e|  \le o(1)  \| \phi\|_{L^\infty(\TO)} + o(\e^{2s-1})  \|\nabla \phi\|_{L^\frac 1{1-s}(\TO)}  \to 0. 
$$ Now, regarding $I_2^\e$ observe that $\nabla u_\e$ is uniformly bounded in $L^\infty$ and as previously observed we can obtain by the embedding of $W^{2s-1, \frac 1s}$ into $L^2$, through \Cref{Christoffel}(iii):
$$
|I_2^\e| \aleq \|\G^\e\|^2_{L^2(\TO)} \|\phi\|_{L^\infty(\TO)}= o(1).
$$  Finally, to finish the proof of our claim, we estimate  once again similar as in \Cref{Curl2nd-e}
\[
|I^\e_3| \aleq  \int_\TO |\Gamma^\e| |\two^\e|  |\phi|  \le   \|\two^\e \|_{L^\frac{2}{s}(\TO)} \| \G^\e\|_{L^\frac{1}{s'}(\TO)} \|\phi\|_{L^\infty(\TO)} \le o(1), \quad \mbox{for} \,\, s\ge \frac 23. \qedhere
\]\end{proof}

\section{Developability}\label{sec-dev}
 We already know by \Cref{compnt-develop} that each component of $u$ is independently developable and has the required regularity. What remains to be shown is that the constancy segments and regions of the developability are the same  for the three components. 
 
 Let $\TO$ be any  smooth bounded domain supported in $\Omega$ and let $f$ be defined as in \Cref{two=nab-f}. We first claim that any such $f$ is developable.  
 
 \begin{proposition}\label{Jac(f)=0}
Let $\Omega, \TO, s, u, f$ be as in \Cref{two=nab-f}.   Then ${\rm Jac}(f)=0$ in $\mathcal D'(\TO)$. In particular since $\nabla f= \two$ is symmetric, the conclusions of  \Cref{th-main1} hold true for $f$. 
\end{proposition}
\begin{proof}
   
We will once again use \Cref{covar}, but this time we will directly pass to the limit as $\e \to 0$. Applying \Cref{weak-product-neg} in view of \Cref{2nd-convergence}, we note that 
$$
\two_{ij}^\e \vec n^{\e, m}  \longrightarrow  \two_{ij}   \vec n^{m}\,\, \mbox{in} \,\, \mathcal D'(\TO) \,\, \mbox{as}\,\,  \e\to 0.  
$$ Also, \Cref{Christoffel}(iii) implies that the first term in the right hand side of \eqref{covar} converges to $0$ in $\mathcal D'(\TO)$. Since $\partial_{ij} u_\e$ converges to $\partial_{ij} u$, we conclude with the following identity for any pair $i,j\in \{1,2\}$:
$$
\partial_{ij} u^m = \two_{ij}   \vec n^{m}.  
$$ Letting $g_m: = \nabla u^m$, this identity reads
\begin{equation}\label{coherence}
\nabla g_m  =  \vec n^{m} \two = \vec n^m \nabla f.
\end{equation} Note that $f,g_m \in W^{s, \frac 2s} (\TO)$ and  $\vec n\in W^{s, \frac 2s} \cap L^\infty (\TO)$. Hence  \Cref{th:lambdajac}  yields that for any 
$\phi \in C^\infty_c(\TO)$
$$
\Jac (g_m)[\phi]= \Jac (f) [(\vec n^m) ^2 \phi]. 
$$  On the other hand by \Cref{compnt-develop} we have ${\rm Jac}(g_m) =0$, therefore for all $\phi \in C^\infty_c(\TO)$
\begin{equation*}
\Jac (f) [\phi] =    \Jac (f) [\sum_{m=1}^3  (\vec n^m)^2 \phi]  =\sum_{m=1}^3 \Jac (f) [(\vec n^m)^2 \phi] = \sum_{m=1}^3 \Jac (g_m) [\phi]  =0.  \qedhere
\end{equation*}  \end{proof}

  We complete the proof of  \Cref{thm-develop}.  We have shown that $f$ is continuous, and for any $x\in \TO$, it is either constant around $x$, or  it  is constant along the connected component of the intersection of a  line passing through $x$ with $\TO$.  By \cite[Corollary 2.9 and Lemma 2.10]{DLP20}, for any $x\in \TO$, there exists a disk $B_x \ni x$ in $\TO$  and Lipschitz unit vector field $\vec \eta$ on $B_x$ such that for all $\psi \in C^\infty_c (B_x)$  
  $$
  (\nabla f)\cdot [\psi \vec \eta] = \int_{B_x} f {\rm div}(\psi \vec \eta) =0. 
  $$  Note that the vector field $\vec \eta$ determines the constancy directions for $f$.  We claim that for each $m$ and  for all $\psi \in C^\infty_c (B_x)$ 
\begin{equation}\label{gm-dev}
  \int_{B_x} {\rm div}(\psi \vec \eta) \nabla u^m   =  0. 
  \end{equation} We remark that proving this claim and  applying \cite[Lemma 2.10 and Proposition  2.1]{DLP20} yields the desired simultaneous constancy of $\nabla u^m$ along the segments defined by $\vec \eta$  and completes the proof of our main theorem.  
  
 To prove  \eqref{gm-dev},  first note that by \Cref{2nd-convergence}-(i)  and \Cref{weak-product-local} we obtain
  $$
  \lim_{\e \to 0} \| \vec n^{\e, m} \nabla f  - \vec n^m \nabla f \|_{\dot{W}^{s-1, \frac 2s}(\TO)}=0, 
  $$  which implies through \eqref{asso-prod}
  $$
(\vec n^m \nabla f) [\psi \vec \eta]  =  \lim_{\e\to 0}    (\vec n^{\e, m} \nabla f) [\psi \vec \eta]  = \lim_{\e\to 0}   \nabla f \cdot  [(\vec n^{\e,m}) \psi \vec \eta] =0.
  $$   Combined with \eqref{coherence} we obtain that for $m=1,2,3$
  $$
 \int_{B_x}{\rm div}(\psi \vec \eta) g_m  =  (\nabla g_m)\cdot [\psi \vec \eta] =  (\vec n^m\nabla f) \cdot [\psi \vec \eta]  =0,
  $$ which establishes \eqref{gm-dev} as claimed.

 \section{Distributional Jacobian determinant  behaves like a determinant}\label{natural}
 
 In this section we will prove \Cref{th:lambdajac}. We first gather some known preliminary results regarding the statement of the theorem.
 
\subsection{Preliminaries}
The following useful lemmas are well-known facts. They can be derived via a tedious argument based on Littlewood-Paley theory and paraproducts \cite{WY99a} which extends an earlier work on the limiting case $s=1$ by \cite{CLMS}.   Much more elegant proofs can be achieved following \cite{BN11}  based on the harmonic extension, see also \cite{BBM05}, and we refer to \cite{LS20} for generalizations. 
\begin{lemma}[Distributional Jacobian]\label{la:distjac}
Let $n\ge 2$, $\Omega \subset \R^n$ be a bounded smooth domain or $\Omega = \R^n$. Assume that $ \frac{n-1}{n} < s <1$, $f \in W^{s,\frac{n}{s}} (\Omega, \R^n)$, $\psi \in W_0^{(1-s)n, \frac{1}{1-s}}(\Omega)$. Then 
\[
 \Jac(f)[\psi] := \lim_{k \to \infty}\Jac(f_k)[\psi_k] 
\]
is well-defined as a distribution in $W^{n(s-1), \frac 1s}$, where $f_k\in C^\infty(\overline \Omega)$ and $\psi_k \in C^\infty_c(\Omega)$ are  any two sequences of  functions converging   to $f$ and $\psi$ in 
their respective norms. 
\end{lemma} 
See, e.g., \cite[Lemma 1.3]{LS19} for a proof. 

\begin{lemma}\cite[Theorem 3.2]{LS20} \label{la:dualspace}
Let $n \geq 1$, $\lambda, g \in W^{\frac{n}{n+1},n+1}(\R^n)$ and  $\phi \in   C^\infty_c (\R^n, \bigwedge^{n-2}(\R^n))$. Then  
\[
\left| \int_{\R^n} \lambda dg \wedge d\phi\right| \aleq [\lambda]_{W^{\frac{n}{n+1},n+1}(\R^n)}\, [g]_{W^{\frac{n}{n+1},n+1}(\R^n)}\, [\phi]_{W^{\frac{2}{n+1},\frac{n+1}{n-1}}(\R^n)}.
\] In particular, by the Stokes theorem for differential forms, and by choosing suitable test forms $\phi$ we have the following estimates for the components:
$$
 \|d(\lambda dg)\|_{\dot{W}^{-\frac{2}{n+1}, \frac{n+1}{2}}(\R^n)} = \|d(\lambda dg)\|_{\Big (\dot{W}^{\frac{2}{n+1},\frac{n+1}{n-1}}(\R^n)\Big )'}  \aleq [\lambda]_{W^{\frac{n}{n+1},n+1}(\R^n)}\, [g]_{W^{\frac{n}{n+1},n+1}(\R^n)}.
$$
\end{lemma}
 
\subsection{A determinant estimate}
\begin{proposition}\label{pr:detest}
 For any $k \in \{0,\ldots,n\}$ and $\phi \in C^\infty_c(\R^n)$,  scalar functions  $a_j \in W^{\frac{n}{n+1},n+1}(\R^n)$ and 1-forms $\beta_j \in W^{\frac{n-1}{n+1},\frac{n+1}{2}} (\R^n, \bigwedge^1 (\R^n))$,
\[ 
\begin{split} 
& \Big | \int_{\R^n}  (da_1\wedge \ldots \wedge da_k \wedge \beta_{k+1} \wedge \ldots \wedge \beta_n) \phi \Big | \\\
 \aleq & \Big (\|\phi\|_{L^\infty} + [\phi]_{W^{\frac{n}{n+1},n+1}} \Big ) \prod_{j=1}^k [a_j]_{W^{\frac{n}{n+1},n+1}} \prod_{j=k+1}^n [\beta_j]_{W^{\frac{n-1}{n+1},\frac{n+1}{2}}}.
\end{split}
 \]
\end{proposition}
 \begin{remark}
 The proposition is indeed  a determinant estimate:
 \[ 
\begin{split} 
& \Big | \int_{\R^n} \det(\nabla a_1,\ldots,\nabla a_k, B_{k+1},\ldots,B_n) \phi \Big | \\\
 \aleq & \Big (\|\phi\|_{L^\infty} + [\phi]_{W^{\frac{n}{n+1},n+1}} \Big ) \prod_{j=1}^k [a_j]_{W^{\frac{n}{n+1},n+1}} \prod_{j=k+1}^n [B_j]_{W^{\frac{n-1}{n+1},\frac{n+1}{2}}}
\end{split}
 \]  for scalar functions and vector fields of appropriate regularity. 
 \end{remark}
 
\begin{proof}
This can be proven by the tedious arguments in \cite{WY99a} using Littlewood-Paley decomposition and paraproducts. Instead we follow an argument inspired by \cite{BN11}, with the  adaptations from \cite{LS20} (see also \cite{Ing20}). Let $a^h$, $\beta^h$, $\phi^h$ be the harmonic extensions of the corresponding forms or vectors to $\R^{n+1}_+$.
\[
\begin{split}
 &\int_{\R^n} (da_1\wedge \ldots \wedge da_k \wedge \beta_{k+1} \wedge \ldots \wedge \beta_n) \phi\\
 =&\int_{\R^{n+1}_+} d\brac{(da^{h}_1\wedge \ldots \wedge da^{h}_k  \wedge \beta^{h}_{k+1} \wedge \ldots \wedge \beta^{h}_n) \phi^h}.
 \end{split}
\]
Since $dd = 0$ we find 
\begin{equation}\label{det-estimate-split}
\begin{split}
 &\abs{\int_{\R^n} (da_1\wedge \ldots da_k \wedge \beta_{k+1} \wedge \ldots \wedge \beta_n) \phi}\\
 \aleq &\sum_{\ell = k+1}^n\int_{\R^{n+1}_+} |D a^{h}_1| \cdot \ldots \cdot |D a^{h}_k| |\beta^{h}_{k+1}| \ldots |D\beta^{h}_\ell| \ldots |\beta^{h}_n| |\phi^h|\\
 &+ \int_{\R^{n+1}_+} |D a^{h}_1| \cdot \ldots \cdot |D a^{h}_k| |\beta^{h}_{k+1}| \ldots |\beta^{h}_n| |D\phi^h|\\
 \end{split}
\end{equation}

Recall that  for the Hardy-Littlewood maximal function $\mathcal{M}$
\begin{equation}\label{maximal}
 |f^h(x,t)| \aleq \mathcal{M}f(x),
\end{equation}
and for $s \in (0,1)$,
\[
[f]_{W^{s,p}} \aeq \brac{\int_{\R^n} \brac{\int_0^\infty |t^{1-\frac{1}{p}-s} D f^h|^p dt} dx}^{\frac{1}{p}}.
\] 
See, e.g., \cite{LS20,Ing20}.
Therefore from H\"older inequality and Sobolev embeddings we obtain for the first terms in \eqref{det-estimate-split}:
\[
\begin{split}
 &\int_{\R^{n+1}_+} |D a^{h}_1| \cdot \ldots \cdot |D a^{h}_k|  |\beta^{h}_{k+1}| \ldots |D\beta^{h}_\ell| \ldots  |\beta^{h}_n| |\phi^h|\\\
 \aleq & \|\mathcal{M} \phi\|_{L^\infty}\Big (\prod_{l=1}^k \|D a^{h}_l\|_{L^{n+1}(\R^{n+1}_+)}\Big)  \|\beta^{h}_{k+1}\|_{L^{n+1}(\R^{n+1}_+)} \ldots 
 \|D\beta^{h}_l\|_{L^{\frac{n+1}{2}}(\R^{n+1}_+)} \ldots \|\beta^{h}_n\|_{L^{n+1}(\R^{n+1}_+)} \\\
 \aleq& \|\mathcal{M} \phi\|_{L^\infty}  (\prod_{l=1}^k [a^{h}_l]_{W^{1,n+1} }\Big)  [\beta^{h}_{k+1}]_{W^{1,\frac{n+1}{2}}} \ldots [\beta^{h}_l]_{W^{1,\frac{n+1}{2}}} \ldots [\beta^{h}_n]_{W^{1,\frac{n+1}{2}}} \\\
 \aleq & \|\mathcal{M} \phi\|_{L^\infty} \Big ( \prod_{l=1}^k [a_l]_{W^{\frac{n}{n+1},n+1}} \Big ) \prod_{l=k+1}^n [\beta_{l}]_{W^{\frac{n-1}{n+1},\frac{n+1}{2}}} 
 \end{split}
\] which is bounded as required in view of \eqref{maximal}.  The last term in \eqref{det-estimate-split} is estimated in the  same manner through a H\"older estimate:
\[
\begin{split}& \int_{\R^{n+1}_+} |D a^{h}_1| \cdot \ldots \cdot |D a^{h}_k| |\beta^{h}_{k+1}| \ldots |\beta^{h}_n| |D\phi^h|\\\
\aleq &    \|D \phi^h\|_{L^{n+1}(\R^{n+1}_+)} \prod_{l=1}^k \|D a^{h}_l\|_{L^{n+1}(\R^{n+1}_+)}  \prod_{l=k+1}^n \|\beta^{h}_l\|_{L^{n+1}(\R^{n+1}_+)}
 \\\
\aleq &  [\phi]_{W^{\frac{n}{n+1},n+1}} \prod_{l=1}^k [a_l]_{W^{\frac{n}{n+1},n+1}} \prod_{l=k+1}^n [\beta_l]_{W^{\frac{n-1}{n+1},\frac{n+1}{2}}}. \qedhere
\end{split}
\]\end{proof}

\subsection{Hodge decomposition}
\begin{proposition}\label{pr:hodge}
Assume that $\lambda \in W^{\frac{n}{n+1},n+1} \cap L^\infty(\R^n)$ and $g \in W^{\frac{n}{n+1},n+1}(\R^n;\R^n)$. Then we can decompose
\[
 \lambda dg = da + \beta,
\]
such that 
\[
[a]_{W^{\frac{n}{n+1},n+1}(\R^n)} \aleq \brac{\|\lambda\|_{L^\infty} + [\lambda]_{W^{\frac{n}{n+1},n+1}(\R^n)}} [g]_{W^{\frac{n}{n+1},n+1}(\R^n)},
\]
\[
[\beta]_{W^{\frac{n-1}{n+1},\frac{n+1}{2}}(\R^n)} \aleq [\lambda]_{W^{\frac{n}{n+1},n+1}(\R^n)} [g]_{W^{\frac{n}{n+1},n+1}(\R^n)}.
\]
\end{proposition}
\begin{proof}
 
 On $\R^n$ we let $\omega:= \lap^{-1} (\lambda dg)$. Hence
\[
 \lap \omega \equiv (dd^\ast + d^\ast d) \omega = \lambda dg.
\]
Set $a := d^\ast \omega$ and $\beta := d^\ast d \omega$. Observe that 
\[
 \lap d \omega = d \lap \omega = d(\lambda dg);
\]
that is,
\[
\beta = d^\ast d\omega = d^\ast \lap^{-1} \brac{d(\lambda dg)}.
\]
Therefore in view of  a component-wise application of \eqref{grad-lap} and \Cref{la:dualspace} we have
\[
\begin{split}
[\beta]_{W^{\frac{n-1}{n+1},\frac{n+1}{2}}(\R^n)} \aleq\,
& \|d (\lambda dg)\|_{\dot{W}^{\frac{n-1}{n+1}-1,\frac{n+1}{2}}(\R^n)}\\
=\,&\|d (\lambda dg)\|_{\dot{W}^{-\frac{2}{n+1},\frac{n+1}{2}}(\R^n)}\\
\aleq\,&[\lambda]_{W^{\frac{n}{n+1},n+1}(\R^n)}\, [g]_{W^{\frac{n}{n+1},n+1}(\R^n)}.
\end{split}
\]
Moreover,
\[
 \lap a = d^\ast \lap \omega = d^\ast (\lambda dg), 
\]
so again
\[
 a = \lap^{-1}d^\ast (\lambda dg).
\]
Using \eqref{grad-lap} as before and \Cref{weak-product}, we obtain as claimed
\[
 [a]_{W^{\frac{n}{n+1},n+1}(\R^n)} \aleq \|\lambda dg\|_{\dot{W}^{\frac{n}{n+1}-1,n+1}(\R^n)} \aleq \brac{\|\lambda\|_{L^\infty} + [\lambda]_{W^{\frac{n}{n+1},n+1}(\R^n)}} [g]_{W^{\frac{n}{n+1},n+1}(\R^n)}. \qedhere
\]
\end{proof}

\begin{proposition}\label{pr:hodgediff}
Assume that $\lambda \in W^{\frac{n}{n+1},n+1} \cap L^\infty(\R^n)$ and $f \in W^{\frac{n}{n+1},n+1}(\R^n;\R^n)$. Then we can decompose
\[
 \lambda_\eps df_\eps - (\lambda df)_\eps = da^\eps + \beta^\eps
\]
such that 
\[
\lim_{\eps \to 0}\,\,  [a^\eps]_{W^{\frac{n}{n+1},n+1}(\R^n)}  =0,
\]
\[
\lim_{\eps \to 0} \,\, [\beta^\eps]_{W^{\frac{n-1}{n+1},\frac{n+1}{2}}(\R^n)}  =0.
\]
\end{proposition}
\begin{proof}
Our arguments are similar to those for \Cref{pr:hodge}. First we consider
\[
\begin{split}
 &\left\|d \big(\lambda_\eps df_\eps - (\lambda df)_\eps\big)\right\|_{\dot{W}^{-\frac{2}{n+1}, \frac{n+1}{2}}(\R^n)}\\
&\leq \left\|d \big(\lambda_\eps df_\eps - \lambda_\eps df\big)\right\|_{\dot{W}^{-\frac{2}{n+1}, \frac{n+1}{2}}(\R^n)} + \left\|d\big(\lambda_\eps df-\lambda df\big)\right\|_{\dot{W}^{-\frac{2}{n+1}, \frac{n+1}{2}}(\R^n)}\\
&+\left\|d\big( \lambda df- (\lambda df)_\e\big)\right\|_{\dot{W}^{-\frac{2}{n+1}, \frac{n+1}{2}}(\R^n)}\\
& =: {\rm I}_\e + {\rm II}_\e + {\rm III}_\e.\\
\end{split}
\]
In view of \Cref{la:dualspace}, we find that 
\[
\begin{split}
{\rm I}_\e + {\rm II}_\e 
&= \left\|d \big(\lambda_\eps d(f_\eps-f)\big)\right\|_{\dot{W}^{-\frac{2}{n+1}, \frac{n+1}{2}}(\R^n)} + \left\|d\big((\lambda_\eps-\lambda) df\big)\right\|_{\dot{W}^{-\frac{2}{n+1}, \frac{n+1}{2}}(\R^n)}\\
&\aleq\,[\lambda_\eps]_{W^{\frac{n}{n+1},n+1}(\R^n)} [f_\eps-f]_{W^{\frac{n}{n+1},n+1}(\R^n)}
 +[\lambda_\eps-\lambda]_{W^{\frac{n}{n+1},n+1}(\R^n)} [f]_{W^{\frac{n}{n+1},n+1}(\R^n)}\\
& \xrightarrow{\eps \to 0}\, 0.
 \end{split}
\]
In addition, we use \Cref{la:dualspace} once again to deduce that $d(\lambda df) \in \dot{W}^{-\frac{2}{n+1}, \frac{n+1}{2}}(\R^n)$. Thus the convolution converges:
\begin{align*}
{\rm III}_\e= \left\|d\big(\lambda df\big)-\Big(d\big(\lambda df\big)\Big)_\eps\right\|_{\dot{W}^{-\frac{2}{n+1}, \frac{n+1}{2}}(\R^n)} \xrightarrow{\eps \to 0} \, 0.
\end{align*}
Putting together the convergence results for ${\rm I}_\e$,  ${\rm II}_\e$, and ${\rm III}_\e$, we arrive at
\begin{equation}\label{eq:hodge:24335}
 \lim_{\eps \to 0} \left\|d \Big(\lambda_\eps df_\eps - \big(\lambda df\big)_\eps\Big)\right\|_{\dot{W}^{-\frac{2}{n+1}, \frac{n+1}{2}}(\R^n)}= 0.
\end{equation}

Now we proceed as in \Cref{pr:hodge}. We first solve on $\R^n$:
\[
 \lap \omega^\e \equiv (dd^\ast + d^\ast d) \omega^\e = \lambda_\eps df_\eps - (\lambda df)_\eps,
\]
and then set $a^\e := d^\ast \omega^\eps$ and $\beta^\eps := d^\ast d \omega^\eps$. Observe that 
\[
 \lap d \omega^\eps = d \lap \omega^\eps = d(\lambda_\eps df_\eps - (\lambda df)_\eps).
\]
That is,\
\[
\beta^\eps = d^\ast d\omega^\eps = d^\ast \lap^{-1} \brac{d\big(\lambda_\eps df_\eps - (\lambda df)_\eps\big)}.
\]
 So, with \eqref{grad-lap} and \eqref{eq:hodge:24335} we find that 
\[
\begin{split}
[\beta^\eps]_{W^{\frac{n-1}{n+1},\frac{n+1}{2}}(\R^n)} \aleq\,& \left\|d(\lambda_\eps df_\eps - (\lambda df)_\eps)\right\|_{\dot{W}^{\frac{n-1}{n+1}-1,\frac{n+1}{2}}(\R^n)}\\
=\,&\left\|d \big(\lambda_\eps df_\eps - (\lambda df\big)_\eps)\right\|_{\dot{W}^{-\frac{2}{n+1}, \frac{n+1}{2}}(\R^n)}\\
&\xrightarrow{\eps \to 0}\,0.
\end{split}
\]

Moreover, we have
\[
 \lap a^\eps = d^\ast \lap \omega^\eps = d^\ast (\lambda_\eps df_\eps - (\lambda df)_\eps), 
\]
so
\[
 a^\eps = \lap^{-1}d^\ast (\lambda_\eps df_\eps - (\lambda df)_\eps).
\]
Once again \eqref{grad-lap}  yields
\[
\begin{split}
 [a^\eps]_{W^{\frac{n}{n+1},n+1}(\R^n)} \aleq\,& \|\lambda_\eps df_\eps - (\lambda df)_\eps\|_{\dot{W}^{\frac{n}{n+1}-1,n+1}(\R^n)}\\
 \aleq\,& \|\lambda_\eps df_\eps- \lambda df\|_{\dot{W}^{\frac{n}{n+1}-1,n+1}(\R^n)} + \|(\lambda df)_\eps - \lambda df\|_{\dot{W}^{\frac{n}{n+1}-1,n+1}(\R^n)}.\\
 \end{split}
\]

We will use \Cref{weak-product} repeatedly throughout the rest of the proof. Observe that $\lambda df \in \dot{W}^{\frac{n}{n+1}-1,n+1}(\R^n)$, so 
\[
 \|(\lambda df)_\eps - \lambda df\|_{\dot{W}^{\frac{n}{n+1}-1,n+1}(\R^n)} \xrightarrow{\eps \to 0} 0.
\]
On the other hand,
\[
 \|\lambda_\eps df_\eps- \lambda df\|_{\dot{W}^{\frac{n}{n+1}-1,n+1}(\R^n)} \leq \|(\lambda_\eps -\lambda)df]\|_{\dot{W}^{\frac{n}{n+1}-1,n+1}(\R^n)} + \|\lambda_\eps d(f_\eps-f)\|_{\dot{W}^{\frac{n}{n+1}-1,n+1}(\R^n)}.
\]
The former term tends to zero as $\e \to 0$. For the latter term, we have
\[
 \|\lambda_\eps d(f_\eps-f)\|_{\dot{W}^{\frac{n}{n+1}-1,n+1}(\R^n)} \aleq \brac{\|\lambda_\eps\|_{L^\infty} + [\lambda_\e]_{W^{\frac{n}{n+1},n+1}(\R^n)}} [f_\eps-f]_{W^{\frac{n}{n+1},n+1}(\R^n)},
\]
which again tends to zero.  \end{proof}

\subsection{Proof of \Cref{th:lambdajac}}
\begin{proof} 
Fix $\phi \in C_c^\infty(\Omega)$. We want to show that
\[
 \Jac(f)[\phi]-\Jac(g)[\lambda^n \phi] = 0.
\]
  
We first boundedly extend $g, \lambda$ on the whole $\R^n$, keeping the same names for convenience. We define $ F:=  \lambda \nabla g$ as a distribution in $\R^n$, which is well-defined by \Cref{weak-product}. Note that  for all  $\eta \in C_c^\infty(\Omega)$, extending $\eta$ by $0$ outside $\Omega$ to $\tilde \eta$,  we obtain by \eqref{eq:distnablaflnablag} in view of \eqref{loc-prod}:  
$$
F[\tilde \eta] =  \nabla f[\eta].
$$  Fix  an open set $\TO \Subset \Omega$  containing $\supp \phi$. For $\e$ small enough, $F_\e := F\ast \varphi_\e$ coincides with $\nabla f_\e$ on 
$\TO$ and hence applying \Cref{la:distjac} we have    
$$
\Jac (f)[\phi] = \lim_{\e\to 0} \int_{\Omega} \det (\nabla f_\e) \phi = \lim_{\e\to 0} \int_{\R^n} \det (F_\e) \phi, 
$$ where $\phi$ is extended by $0$ outside $\Omega$ to $\R^n$. Also, mollifying $g$ and $\lambda$  and once again applying  \Cref{la:distjac} we obtain
$$ \Jac (g)[\lambda^n \phi]   =  \lim_{\e \to 0}  \int_{\Omega}  \det (\nabla g_\e)  \lambda_\e ^n  \phi = \lim_{\e \to 0}  \int_{\R^n}  \det (\nabla g_\e)  \lambda_\e ^n  \phi,
$$   since  $\lambda^n_\e \phi \to \lambda^n \phi$ in $W_{00}^{(1-s)n, \frac {1}{1-s}} (\Omega)$.   Therefore we have
\begin{align*}
&\Jac(f)[\phi]-\Jac(g)[\lambda^n \phi]\\
&\quad =  \lim_{\eps \to \infty} \int_{\R^n} \Big ( \det(F_\eps)-\det(\lambda_\eps \na g_\eps)\Big ) \phi 
=  \lim_{\eps \to \infty} \int_{\R^n} \Big(\det((\lambda \na g)_\eps)-\det(\lambda_\eps \na g_\eps)\Big)  \phi\\
&\quad = \sum_{j=1}^n \int_{\R^n}  \left ( \lambda_\e d g^1_\e \wedge \cdots \wedge \lambda_\e d g_\e^{j-1} \wedge \Big[ (\lambda d g^j)_\e - \lambda_\e d g^j_\e \Big] \wedge (\lambda d g^{j+1})_\e \wedge \cdots \wedge (\lambda d g^{n})_\e \right)  \phi.
\end{align*}
In view of the Sobolev embedding
$$
W^{s,\frac sn}(\R^n) \hookrightarrow W^{\frac {n}{n+1}, n+1} (\R^n),
$$ for $s\ge \frac {n}{n+1}$, and the fact that the distributional identity in the bigger space implies the one in the smaller space, we can assume that $s= \frac {n}{n+1}$. For each entry of the form $(\lambda d g^i)_\eps$ and $\lambda_\eps d g^i_\eps$, we shall apply Hodge decomposition as in \Cref{pr:hodge}. To the difference term $(\lambda d g^j)_\eps-\lambda_\eps d g^j_\eps$ we apply Hodge decomposition as in \Cref{pr:hodgediff}. We then obtain terms of the form:
\[
 \int_{\R^n} \left (da^\eps_1\wedge \ldots \wedge da^\eps_k \wedge \beta^\eps_{k+1} \wedge \ldots \wedge \beta^\eps_n\right) \phi,
\]
where each $a^\eps_j$ and $\beta^\eps_j$ is bounded in its corresponding semi-norm. Note that,  fixing $\e$, the estimates in  \Cref{pr:detest} are still valid for the above integral since  by construction we can approximate  each  $a^\eps_j$ (resp.\@ $\beta^\eps_j$) in its semi-norm   by a sequence of scalar functions in $W^{\frac{n}{n+1},n+1}(\R^n)$  (resp.\@ 1-forms  in $W^{\frac{n-1}{n+1},\frac{n+1}{2}} (\R^n, \bigwedge^1 (\R^n))$.  Therefore, to conclude, we use \Cref{pr:detest}: one of the term $a^\eps_j$ or one of $\beta^\eps_j$ converges to zero (since it comes from the difference term), in the corresponding norm, thanks to \Cref{pr:hodgediff}, while the other terms are bounded by \Cref{pr:hodge}. So we obtain the claim by taking $\eps \to 0$.  \end{proof}

\appendix
\renewcommand{\thesection}{\Roman{section}} 
\section{A proof of Proposition~\ref{Jac(f)=0} for \texorpdfstring{$s>2/3$}{s>2/3}}\label{weaker-proof}

As a tangential note, in this section we will   sketch how a slightly weaker statement than \Cref{Jac(f)=0} can be obtained without using \Cref{th:lambdajac}. This hence provides another proof of  \Cref{thm-develop}, but only for $s>2/3$. Hereby, we would  like to highlight the importance of  \Cref{th:lambdajac} in completing our proof for the critical case $s=2/3$.

We begin first by the following observation. As a corollary of the gained regularity $u\in c^{0, \frac s2}$ in \Cref{compnt-develop}, we  can improve some of the estimates  of previous sections and prove: 

 \begin{proposition}\label{better-convergences}
 Let $\Omega \subset \R^2$ be a bounded smooth domain, $\frac 23\le s<1$ and $u\in  I^{1+s, \frac2s} (\Omega, \R^3)$. Let $\theta \in [0,1]$.  For all $\TO \Subset \Omega$
 
 \begin{itemize}
 \smallskip
\item[{\rm (i)}] $\|\two^\e\|_{L^{\frac {2}{s\theta}}(\TO)} \le o(\e^{\frac s2 (1+ \theta)-1})$.  \smallskip
\item[{\rm (ii)}] $\|\G^\e\|_{L^{\frac 1{s\theta}}(\TO)} \le o(\e^{ s (1+ \theta)-1})$.   \end{itemize}
 \end{proposition}

\begin{proof}
The estimates are obtained by interpolating the estimates in \Cref{Christoffel} with a new set of estimates obtained through $c^{0,\frac s2}$ regularity in the same manner; see \cite[Equations (4.4) and (4.8)]{DLP20}. We will leave the details to the reader.    
\end{proof}

An immediate corollary is the following better than $L^1$-estimate for ${\rm curl} \, \two^\e$. As we previously explained in \Cref{L1-not-enough}, this is the missing link  for following the steps of \cite{DLP20} in proving our main theorem.  We can now obtain this estimate  only for the super-critical values of $s>2/3$. 
\begin{corollary}
If $s>\frac 23$, there exists $r>1$ such that 
$$
\lim_{\e \to 0}  \|{\rm curl} \, \two^\e\|_{L^r(\TO)} =0.
$$
\end{corollary}
\begin{proof}
Letting 
$$
\frac 1r = \frac {s\theta}{2} + s\theta = \frac {3s}2\theta,
$$  we have  
$$
\|{\rm curl}\, \two^\e\|_{L^r(\TO)} \le \|\two\|_{L^{\frac {2}{s\theta}}} \|\G^\e\|_{L^{\frac 1{s\theta}}}  \le o(\e^{\frac {3s}{2} (1+ \theta) - 2}). 
$$ To complete the proof we need to show that there is $\theta \in (0,1)$ such that
$$
r >1 \quad \mbox{and} \quad \frac {3s}{2} (1+ \theta) - 2 \ge 0.
$$ These are respectively equivalent to 
$$ 
\theta  <\frac 2{3s}  \quad \mbox{and} \quad \theta \ge \frac {4}{3s} -1.
$$ But if $\frac 23 <s  < 1$  we have
$$
0<\frac 13  <\frac {4}{3s} -1  <  \frac 2{3s}  <1, 
$$ and so we can choose any $\theta \in [\frac {4}{3s} -1,  \frac 2{3s})$.
 \end{proof}
 
 Once the $L^r$ vanishing estimate for ${\rm curl}\, \two^\e$ is obtained, and having the usual elliptic estimates at hand, one can proceed as in \cite[Proposition 4.5]{DLP20} to show that $\Jac(f) \equiv 0$ as required by \Cref{Jac(f)=0}.  This completes the proof of \Cref{thm-develop} but only for $s>2/3$ as in \Cref{sec-dev}. Once again, we will leave the details to the interested reader. 
 
\section{Fractional Absolute Continuity}\label{s-ac}

In proving \Cref{th-main1},  we  used the following result, which  follows by an embedding theorem from a known result for Bessel-potential spaces \cite[Theorem 1.1]{HH15}. 
 
\begin{theorem}\label{th-hausdorffcontentzero}
Let $u \in W^{s,p}(\R,\R^m)$ with $s \in (0,1)$, $p \in (1,\infty)$ such that and $sp > 1$ and let $I$ be a finite interval. Then the Hausdorff dimension $\mathcal{H}$-dim of $u^*(I) \leq \frac{1}{s}$ for any interval $I \subset \R$. Here $u^*$ denotes the continuous representative of $u$.
\end{theorem}
Indeed, following \cite[Theorem 7.63 (g)]{adams}, we note that for any $p>1$ and $\eps>0$, 
$$
W^{s,p} (\R^n)   \hookrightarrow   L^{s-\eps,p} (\R^n).  
$$ Choosing $\eps>0$ such that $p(s-\eps)-1>0$, and applying   \cite[Theorem 1.1]{HH15}, we obtain \Cref{th-hausdorffcontentzero}. (Note the notational disparity  with \cite{HH15}, which uses $W^{s,p}$ for the Bessel-potential space $H^s_p = L^{s,p}$.)
 
\begin{remark}
\begin{itemize}
 \item[] The typical space-filling curves provide counterexamples to \Cref{th-hausdorffcontentzero} if $sp < 1$. E.g. the Peano-curve $f: I \to \R^2$ that fills a square is of class $C^{1/2}$, and thus belongs to $W^{s,2}$ for any $s < \frac{1}{2}$ -- however $\mathcal{H}^2_{\infty}(f(I)) \neq 0$.
 
 \item[] The case $sp = 1$ is quite curious. It is known that for $u \in W^{1,1}(I,\R^N)$, if $u^\ast$ denotes its continuous representative then $\mathcal{H}^1(u^\ast(I)) < \infty$. This is also based on the absolute continuity of the integral, however, in the fractional case $s<1$ the condition $sp=1$ does not guarantee continuity in one dimension. Indeed, it is unclear to us if  there is always a representative $u^\ast$ for  $u \in W^{s,\frac{1}{s}}(\R,\R^N)$ such that $\mathcal{H}^{\frac{1}{s}}(u^\ast(I)) < \infty$.
\end{itemize}
\end{remark}

We would like to note that \Cref{th-hausdorffcontentzero} also follows from a notion reminiscent of absolute continuity for fractional Sobolev maps. It is well-known that \Cref{th-hausdorffcontentzero} holds for $s=1$ and $p>1$, which is a consequence of absolute continuity of $W^{1,1}$-maps. Also it is known from the area formula and the Luzin property \cite[Lemma 21]{H00} that the continuous representative of a map $u \in W^{1,p}(\R^n,\R^m)$ for $m \geq n \geq 2$ and $p > n$ has image $\mathcal{H}^p(u(\R^n)) = 0$. In this sense, \Cref{th-hausdorffcontentzero} is a natural extension to maps with one-dimensional domain in fractional Sobolev spaces.   In this appendix we will further discuss this approach. The authors do not know of any instance in the literature where the following observations are made.

One of the basic Sobolev space results is that the continuous representative $f^\ast$ of a function $f \in W^{1,1}$ is absolutely continuous, that is for any $\eps > 0$ there exists $\delta > 0$ such that whenever we have a pairwise disjoint collection of intervals $(I_i)_{i=1}^\infty$ with
\[
 \sum_{i} |I_i| < \delta
\] then
\[
 \sum_{i} |f^\ast(x)-f^\ast(y)| < \eps.
\] This follows easily from the fundamental theorem of calculus (which holds for the continuous  representative $f^\ast$)
\[
 f^\ast(a)-f^\ast(b) = \int_{a}^b f'(z)\, dz
\]
and the absolute continuity of the integral, which says that if $g \in L^1(\Omega)$ then for any $\eps > 0$ there exists $\delta > 0$ such that 
\[
 \|g\|_{L^1(U)} < \eps \quad \forall U \subset \Omega \text{ measurable }: |U| < \delta.
\]
By a covering argument, it is also easy to show that an absolutely continuous function $f: I \subset \R \to \R^N$ must have a 1-dimensional finite Hausdorff content $\mathcal{H}^1_\infty(f(I)) < \infty$, where
\[
 \mathcal{H}^p_\infty(A) := \inf \left \{ \sum_{i} (r_i)^p: \ \text{there is a cover of $A \subset \bigcup_{i} B(r_i)$ with balls $B(r_i)$ of radius $r_i > 0$} \right \}.
\] The underlying reason for \Cref{th-hausdorffcontentzero} is that there is a fractional generalization of a sort of absolute continuity to fractional Sobolev spaces $W^{s,p}(\R)$ as long as $sp> 1$. Observe that for $s < 1$ there are discontinuous functions in $W^{s,p}$ with $sp=1$. 
\begin{definition}[(t,p)-absolute continuity]
Let $t\ge 0$ and $p\in (0,\infty)$. A continuous function $f: \R \to \R^N$ is called $(t,p)$-absolutely continuous if the following holds.
For any $\epsilon > 0$ there exists a $\delta > 0$ such that 
whenever we have a disjoint intervals $(I_i)_{i=1}^\infty$ with
\[
 \sum_{i} |I_i| < \delta
\]
then
\[
 \sum_{i} \sup_{x\neq y \in I_i} \frac{|f(x)-f(y)|}{|x-y|^t}^p < \epsilon.
\]
\end{definition}
For $p=1$, $t=0$ this is the usual absolute continuity.

The following lemmas are elementary.
\begin{lemma}\label{abs-monotone}
If  $\frac{1+ \tilde t}{1+t} \le \frac{\tilde p}{p} \le 1 $, then $(t,p)$-absolute continuity implies $(\tilde t, \tilde p)$-absolute continuity.  
\end{lemma}
\begin{proof}
Let $\lambda:= \tilde p/p \le 1$.  For any collection of  disjoint intervals $I_i$ we have
\[
\begin{aligned}
\sum_{i} \sup_{x\neq y \in I_i} \frac{|f(x)-f(y)|}{|x-y|^{\tilde t}}^{\tilde p}&  = \sum_{i} \sup_{x\neq y \in I_i} \Big (\frac{|f(x)-f(y)|}{|x-y|^{t}}^{p}\Big )^\lambda |x-y|^{\lambda t- \tilde t}\\ & \le \sum_{i} \Big ( \sup_{x\neq y \in I_i}  \frac{|f(x)-f(y)|}{|x-y|^{t}}^{p}\Big )^\lambda |I_i|^{\lambda t- \tilde t} 
\\ & \le \Big ( \sum_{i}  \sup_{x\neq y \in I_i}  \frac{|f(x)-f(y)|}{|x-y|^{t}}^{p}  \Big )^\lambda \Big (\sum_i |I_i|^{\frac{(\lambda t- \tilde t)}{1-\lambda}} \Big )^{1-\lambda},
\end{aligned}
\] where we used the H\"older inequality $\|\cdot\|_{l^1}\le \|\cdot\|_{l^\frac1\lambda} \|\cdot\|_{l^{\frac{1}{1-\lambda}}}$. Since $f$ is $(t,p)$-absolutely continuous, given $\epsilon>0$, we choose $\delta_1>0$ such that 
$$
  \sum_{i}  \sup_{x\neq y\in I_i}  \frac{|f(x)-f(y)|}{|x-y|^{t}}^{p}  <\epsilon.
$$ Note that by the assumption 
$$
\frac{\lambda t- \tilde t}{1-\lambda} \ge 1. 
$$ If $\displaystyle \sum_{i} |I_i| < \delta:= \min\{\delta_1, \epsilon \}$, we hence obtain by combining the  above estimates 
\[
\begin{aligned}
\sum_{i} \sup_{x\neq y \in I_i} \frac{|f(x)-f(y)|}{|x-y|^{\tilde t}}^{\tilde p}& < \epsilon^\lambda \Big ( \sum_i |I_i|\Big )^{\frac{\lambda t- \tilde t}{1-\lambda}(1-\lambda)} 
< \epsilon^\lambda \delta ^{ \lambda t-\tilde t}  \le \epsilon.  
\end{aligned}\qedhere
\] 
\end{proof}

\begin{lemma}[Hausdorff content of $(t,p)$-absolutely continuous maps]\label{la-abscontcover}
Let $f: I \to \R^N$ be $(t,p)$-absolutely continuous. Then if $t> 0$
\[
 \mathcal{H}^{p}_\infty(f(I)) =0.
\]
If $t = 0$ we still have
\[
 \mathcal{H}^{p}_\infty(f(I)) < \infty.
\]
\end{lemma}
\begin{proof}
In the definition of $(t,p)$-absolute continuity let $\epsilon = 1$ and obtain some $\delta > 0$. Let $\tilde{I}$ be any subinterval of $I$ with $\diam(\tilde{I}) < \frac{\delta}{2}$.  For any $\sigma > 0$ we find $N=N(\sigma)$-finitely many intervals $I_i$ and $J_i$ such that each $(I_i)_{i=1}^N$ and $(J_i)_{i=1}^N$ are pairwise disjoint, $|I_i|, |J_i| < \sigma$ and $\bigcup_{i} I_i \cup J_i = \tilde{I}$. Each $f(I_i)$ (resp.\@ $f(J_i)$) is then contained in a ball of radius $2\sigma^{\frac{t}{p}} \sup_{x,y \in I_i} \frac{|f(x)-f(y)|}{|x-y|^{\frac{t}{p}}}$ (centered at $f(x_i)$ for some $x_i \in I_i$ (resp.\@ $x\in J_i$). By $(t,p)$-absolute continuity we then have
\[
\mathcal{H}^{p}_{\infty} (f(\tilde{I}) ) \aleq \sum_{i} \sigma^{t} \brac{\sup_{x,y \in I_i} \frac{|f(x)-f(y)|^p}{|x-y|^{t}}+\sup_{x,y \in J_i} \frac{|f(x)-f(y)|^p}{|x-y|^{t}}} \aleq \sigma^t.
\]
Since this holds for any subinterval $\tilde{I}$ of diameter $\frac{\delta}{2}$, we cover $I$ by $\approx \frac{1}{\delta}$ many such intervals and obtain
\[
 \mathcal{H}^{p}_{\infty} (f(I)) \aleq \frac{1}{\delta} \sigma^t < \infty.
\]
If $t>0$ we can take $\sigma$ arbitrarily small to obtain $\mathcal{H}^{p}_{\infty} (f(I)) = 0$. 
\end{proof}

In view of the above two lemmas, \Cref{th-hausdorffcontentzero}  will follow from one last statement .   

\begin{lemma}\label{th-wspabscont}
Let $s \in (0,1)$, $p \in (1,\infty)$ with $sp > 1$. Then the continuous representative $u^\ast$ of any map $u \in W^{s,p}(\R)$ is $(sp-1,p)$-absolutely continuous.
\end{lemma}
\begin{remark}
\begin{itemize}
\item[] For $s=1$ and $p=1$ the result is still true (and it is the classical absolute continuity result for $W^{1,1}$-maps in $1$ dimension).
\item[] There cannot be such a result for when $sp=1$, $s < 1$, since $W^{s,\frac{1}{s}}$ does not embed  into the continuous functions. E.g., $s=\frac{1}{2}$ and $p=2$: denote $B^2 \subset \R^2$ the unit ball in $\R^2$ and $B^2_+ := B^2 \cap \R^2_+$  the upper halfball then $\log \log 2\sqrt{(x_1)^2+(x_2)^2}$ belongs to $W^{1,2}(B^2)$, thus to $W^{1,2}(B^2_+)$. By trace theorem for $I=[-1/2,1/2]$, we find that $\log \log 2 |x_1| \in W^{\frac{1}{2},2}(I)$, however this is clearly not a continuous function (let alone absolutely continuous of any sense).
\end{itemize}
\end{remark}

\begin{proof}[Proof of \Cref{th-wspabscont}]
  Since $sp>1$, $W^{s,p}(I)$ embeds in $C^{0,s-1/p}(I)$ for any interval (see e.g.\@ \cite[Section 8]{DNPV12}).  Indeed,  for a universal constant $C>0$ and all $a,b\in I$ we have
  $$
 |u^\ast(b)- u^\ast(a)| \le   C [u]_{W^{s,p}(I)} |a-b|^{s-1/p} , 
 $$ which gives for $a\neq b$:
 $$
 \frac{|u^\ast(b)- u^\ast(a)|^p}{|a-b|^{sp-1}}\le  C [u]^p_{W^{s,p}(I)}.
 $$ We therefore obtain for the mutually disjoint $I_i$:
$$
\sum_{i}  \sup_{x\neq y \in I_i} \frac{|u^\ast(x)-u^\ast(y)|^p}{|x-y|^{sp-1}}\leq C \sum_{i}  [u]^p_{W^{s,p}(I_i)} \le   [u]^p_{W^{s,p}(A)},
 $$  where $\displaystyle A = \bigcup I_i$. Now, by the absolute continuity of the integral $[u]^p_{W^{s,p}(\R)} <\infty$, for any $\epsilon>0$,  there is $\delta$ small enough such that $|A|= \displaystyle \sum_{i}|I_i| <\delta$ implies  $[u]^p_{W^{s,p}(A)} < \epsilon$.
\end{proof}

\begin{proof}[Proof of \Cref{th-hausdorffcontentzero}]

By \Cref{th-wspabscont}, $f$ is $(sp-1, p)$-absolutely continuous. Let $\tilde p = 1/s$ and $\tilde s= s$. Then 
$$
\frac {1+ \tilde s\tilde p  - 1  }{1+ sp-1} = \frac {\tilde p}p = \frac{1}{sp} <1. 
$$ The conditions of \Cref{abs-monotone} are satisfied and hence $f$ is also   $(\tilde s\tilde p-1, \tilde p)$-absolutely continuous. 
\Cref{la-abscontcover} implies that the Hausdorff dimension of $u^\ast(I)$ is at most $\tilde p= 1/s$, as required.  
\end{proof}
\begin{remark}We could have also used the Sobolev embedding  $W_{loc}^{s,p}(\R,\R^m) \hookrightarrow W^{\tilde{s},\tilde{p}}_{loc}(\R,\R^m)$ for any $\tilde{s} < s$ and $\tilde{p} \leq p$ \cite[Proposition  2.1.2 and Theorem 2.4.4/1]{RuSi}, but note that this is not necessarily true for $\tilde{s} = s$ \cite{MS15}, and some small adjustment would have become necessary.
 \end{remark}
 
\bibliographystyle{abbrv}%
\bibliography{bib}%

\end{document}